\documentclass{article}%
\usepackage{authblk}
\usepackage{amsmath}%
\usepackage{amssymb}%
\usepackage{stix}
\usepackage{graphicx}
\usepackage{hyperref}
\usepackage[usenames,dvipsnames]{color}
\newcommand{\abbr}[1]{{\sc\lowercase{#1}}}

\providecommand{\U}[1]{\protect\rule{.1in}{.1in}}
\newtheorem{theorem}{Theorem}

\newtheorem{condition}[theorem]{Condition}

\newtheorem{corollary}[theorem]{Corollary}

\newtheorem{definition}[theorem]{Definition}

\newtheorem{lemma}[theorem]{Lemma}

\newtheorem{proposition}[theorem]{Proposition}
\newtheorem{remark}[theorem]{Remark}

\newenvironment{proof}[1][Proof]{\noindent\textbf{#1.} }{\ \rule{0.5em}{0.5em}}

\newcommand{\B}{\mathbb{B}}

\newcommand{\E}{\mathbb{E}}

\newcommand{\N}{\mathbb{N}}

	\renewcommand{\P}{\mathbb{P}}

\newcommand{\R}{\mathbb{R}}

\newcommand{\wt}{\widetilde}
\newcommand{\ovl}{\overline}
\newcommand{\ep}{\epsilon}
\newcommand{\cD}{\mathcal{D}}

\newcommand{\cL}{\mathcal{L}}
\newcommand{\cM}{\mathcal{M}}

\newcommand{\cO}{\mathcal{O}}
\newcommand{\cP}{\mathcal{P}}

\begin{document}
\title{The KPP equation as a scaling limit \\
of locally interacting Brownian particles
}
\author{Franco Flandoli and Ruojun Huang}
\date{}
\maketitle

\begin{abstract}
Fisher-KPP equation is proved to be the scaling limit of a system of Brownian
particles with local interaction. Particles proliferate and die depending on
the local concentration of other particles. Opposite to discrete models,
controlling concentration of particles is a major difficulty in Brownian
particle interaction; local interactions instead of mean field or moderate
ones makes it more difficult to implement the law of large numbers properties. The approach taken here to overcome these difficulties is largely inspired by A.
Hammond and F. Rezakhanlou \cite{HR} implemented there in the mean free path case
instead of the local interaction regime. 
\end{abstract}

\section{Introduction}
We consider the scaling limit (in a ``local" interaction regime) of the empirical measure of an interacting Brownian particle system in $\R^d$, $d\ge 1$, describing the proliferation mechanism of cells, where only approximately a constant number of particles interact with a given particle at any given time. We connect the evolution of its empirical measure process to the Fisher-KPP (\abbr{\abbr{F-KPP}}) equation. 

The \abbr{\abbr{F-KPP}} equation is related to particle systems in several ways. One of
them is the probabilistic representation by branching processes, see \cite{McK}
which originated a large literature. Others have to do with scaling limits of
interacting particles. In the case of discrete particles occupying the sites
of a lattice, with local interaction some of the main classical works are \cite{DFL, DP, FLT}; see also the recent contributions \cite{DFPV, FT}.

The discrete setting, as known for many other systems (see for instance \cite{KL}) offers special opportunities due to the simplicity of certain invariant or
underlying measures, often of Bernoulli type; the technology in that case has
become very rich and deeply developed. Different is the case of interacting
diffusions, less developed. The mean field theory, for diffusions, is a
flexible and elegant theory \cite{Szn} but localizing the interactions is very
difficult, see for instance \cite{Va, Uch} as few of the attempts. When the
interaction is {\it moderate}, namely intermediate between mean field and local, there are
more technical opportunities, widely discovered by K. Oelschl\"ager in a series
of works including \cite{Oel2, Oel}. Along these lines, also the
\abbr{\abbr{F-KPP}} equation has been obtained as a scaling limit in \cite{FLO}. Let
us mention also \cite{MR, BM, Met, NO, Ste}  for
related works. 

In the present work we fill in the gap and prove that the \abbr{\abbr{F-KPP}} equation
is also the scaling limit of diffusions, locally interacting. In a sense, this
is the analog of the discrete particle results of \cite{DFL} and the other references
above. The proof is not based on special reference measures of Bernoulli type
as in the discrete case (not invariant in the present proliferation case, but
still fundamental), but it is strongly inspired by the work of \cite{HR}, which
deals with locally interacting diffusions in the mean-free-path regime, that
we adapt to the local regime (the former requires that a particle meets a
finite number of others, on average, in a unit amount of time; the latter
requires that a particle has a finite number of others, on average, in its
own neighborhood, where interaction takes place). Compared to the discrete
setting \cite{DFL}, where the dynamics is a superposition of simple-exclusion process
(which leads to the diffusion operator) and spin-flip dynamics (leading to
the reaction term) and the number of particles per site is either zero or one, we
have to worry about concentration of particles, one of the main difficulties
for the investigation of interacting diffusions. 

After this short introduction to the subject, let us give some more technical
details. We start with the formal definition of a closely-related model already studied in \cite{FLO} in the so-called moderate interaction regime. Then we introduce our slightly altered model.
\begin{definition}
A configuration of the system is a sequence
\[
\eta=\left(  x_{i},a_{i}\right)  _{i\in\mathbb{N}}\in\left(  \mathbb{R}%
^{d}\times\left\{  L,N\right\}  \right)  ^{\mathbb{N}}%
\]
with the following constraint: there exists $i_{\max}\left(  \eta\right)
\in\mathbb{N}$ such that $a_{i}=L$ for $i\leq i_{\max}\left(  \eta\right)  $,
$a_{i}=N$ and $x_{i}=0$ for $i>i_{\max}\left(  \eta\right)  $. 
\end{definition}

The heuristic meaning is that particles with index $i\leq i_{\max}\left(
\eta\right)  $ exist, are alive ($=L$), and occupy position $x_{i}$; particles
with $i>i_{\max}\left(  \eta\right)  $ do not exist yet ($=N$), but we
formally include them in the description; they are placed at $x_{i}=0$.

Test functions $F$ are functions on $\left(  \mathbb{R}^{d}\times\left\{
L,N\right\}  \right)  ^{\mathbb{N}}$ which depend only on a finite number of
coordinates, $F=F\left(  x_{1},...,x_{n},a_{1},...,a_{n}\right)  $ with
$\left(  x_{i},a_{i}\right)  \in\mathbb{R}^{d}\times\left\{  L,N\right\}  $
and are smooth in $\left(  x_{1},...,x_{n}\right)  \in\mathbb{R}^{dn}$. 

\begin{definition}
The infinitesimal generator $\mathcal{L}_{N}$, parametrized by $N\in
\mathbb{N}$, is given by%
\begin{align}\label{gen-original}
\left(  \mathcal{L}_{N}F\right)  \left(  \eta\right)  =\sum_{i\leq i_{\max
}\left(  \eta\right)  }\frac{1}{2}\Delta_{x_{i}}F\left(  \eta\right)
+\sum_{j\leq i_{\max}\left(  \eta\right)  }\lambda_{N}^{j}\left(  \eta\right)
\left[  F\left(  \eta^{j}\right)  -F\left(  \eta\right)  \right]
\end{align}
where, if $\eta=\left(  x_{i},a_{i}\right)  _{i\in\mathbb{N}}$, then $\eta
^{j}=(  x_{i}^{j},a_{i}^{j})  _{i\in\mathbb{N}}$ is given by
\begin{align*}
(  x_{i}^{j},a_{i}^{j})    & =\left(  x_{i},a_{i}\right)  \text{
for }i\neq i_{\max}\left(  \eta\right)  +1\\
\big(  x_{i_{\max}\left(  \eta\right)  +1}^{j},a_{i_{\max}\left(
\eta\right)  +1}^{j}\big)    & =\left(  x_{j},L\right)  .
\end{align*}
The rate $\lambda_{N}^{j}\left(  \eta\right)  $ is given by
\begin{align}\label{rate-original}
\lambda_{N}^{j}\left(  \eta\right)  =\Big(  1-\frac{1}{N}\sum_{k\leq i_{\max
}\left(  \eta\right)  }\theta_{N}\left(  x_{j}-x_{k}\right)  \Big)  ^{+}%
\end{align}
where $\theta_{N}$ are smooth compact support mollifiers with a rate of
convergence to the delta Dirac at zero specified in the sequel. 
\end{definition}

The heuristic behind this definition is that: i) existing particles move at
random like independent Brownian motions;\ ii) a new particle could be created at the position of an existing particle $j$, with rate proportional to the empty space in a neighborhood of $x_{j}$, neighborhood described by the support of
$\theta_{N}$. Our aim is to choose the scaling of $\theta_{N}$, namely the neighborhood of interaction, such that only a small finite number of particles different from $j$ are in that neighborhood. 

\medskip

In the classical studies of continuum interacting particle systems, where interactions are modulated by a potential, one usually takes 
\begin{align*}
\theta_N(x)=N^{\beta}\theta(N^{\beta/d}x)
\end{align*}
for some smooth compactly supported function $\theta(\cdot)$, where $N$ is the order of the number of particles in the system. The case $\beta=0$ is called mean-field, since all particles interact with each other at any given time. The case $\beta\in(0,1)$ is called moderate, as not all particles interact at any given time, nevertheless such number is diverging with $N$. The case $\beta=1$ is called local, as one would expect that in a neighborhood of radius $N^{-1/d}$, only a constant number of particles interact. Of course, here we are assuming that particles are relatively homogeneously distributed in space at all times down to the microscopic scale (which is not always proven). For the system with generator \eqref{gen-original}, the moderate scaling regime with $\beta\in(0,1/2)$ has been studied and its scaling limit to \abbr{F-KPP} equation established in \cite{FLO}, with earlier results \cite{Oel} for a shorter range of $\beta$, and so is the mean-field case whose limit is a different kind of equation \cite{CM,FL}; our aim here is to study the local regime, subject to a modification of the rate \eqref{rate-original}. 

We introduce a scale parameter $\ep\in(0,1]$ describing the length scale where two particles can interact. In particular, in the local regime, $\ep=N^{-1/d}$, but our result is more general, covering also the whole moderate regime (\cite{FLO, Oel}). Then we consider the mollifier 
\begin{align*}
\theta^\ep(x):=\ep^{-d}\theta(\ep^{-1}x)
\end{align*}
built from a given, nonnegative, H\"older continuous and compactly supported function $\theta:\R^d\to\R_+$ with $\int \theta =1$. The rate of proliferation \eqref{rate-original} can be written as 
\begin{align}\label{rate}
\lambda^j_N(\eta)=\Big(1-\frac{1}{N}\sum_{k\le i_{\text{max}}(\eta)}\theta^\ep(x_j-x_k)\Big)^+
\end{align}
whereby the proliferation part of the generator is
\begin{align*}
(\cL_CF)(\eta)&:=\sum_{j\le i_{\text{max}}(\eta)}\Big[1-\frac{1}{N}\sum_{k\le i_{\text{max}}(\eta)}\theta^\ep(x_j-x_k)\Big]^+[F(\eta^j)-F(\eta)].
\end{align*}
Throughout the paper any sum is only over particles alive in the system (whose cardinality is always finite), hence we do not discuss the label $a_j$ of particle $j$. 

Heuristically, the positive part on the rate \eqref{rate} should be insignificant, as we would guess that if starting with a density profile not larger than $1$, then subsequently the density of particles everywhere is no larger than $1$. This is the case for the \abbr{F-KPP} equation (see \eqref{fkpp} below). However, at the microscopic level we do not have effective control on the scale of $\ep$, even a posteriori. Hence in this paper we consider a slightly altered model, namely one without the positive part in the rate. 

Note that now the proliferation rate can be negative, which we will interpret to mean, in terms of the proliferation part of the generator,
\begin{align}\label{pr-ge}
(\wt\cL_CF)(\eta)&=\sum_{j\le i_{\text{max}}(\eta)}[F(\eta^j)-F(\eta)]+\frac{1}{N}\sum_{j, k\le i_{\text{max}}(\eta)}\theta^\ep(x_j-x_k)[F(\eta^{-j})-F(\eta)],
\end{align}
where $\eta^{-j}$ signifies deleting particle $j$ from the collection $\eta$. Thus, the infinitesimal generator of our particle system under study is
\begin{align}
(\wt \cL_N F)(\eta)=\sum_{j\le i_{\text{max}}(\eta)}\frac{1}{2}\Delta_{x_i}F(\eta)+(\wt\cL_CF)(\eta).
\end{align}

\begin{condition}\label{kpp-init}
The function $u_0$ appearing as initial condition satisfies:\\
(a). It is compactly supportly in $\B(0,R)$, an open ball of radius $R$ around the origin.\\
(b). $0\le u_0(x)\le \gamma$ for some finite constant $\gamma$ and all $x\in\R^d$.
\end{condition}
In particular, $u_0\in L^1_+(\R^d)$ (the space of nonnegative integrable functions), with $\|u_0\|_{L^1}\le \gamma |\B(0, R)|$. Denoting by $\eta(t)$ the collection of alive particles at time $t\ge 0$, and $N(t)$ its cardinality, we distribute at time $t=0$, 
\[
i_{max}(\eta(0))=N_0:=N\int_{\R^d} u_0
\]
number of points independently with identical probability density $(\int u_0)^{-1} u_0$ on $\R^d$, for some $u_0$ satisfying Condition \ref{kpp-init}. In particular, $u_0(x)dx$ is the weak limit, in probability, of the initial (normalized) empirical measure:
\begin{align}\label{initial}
\frac{1}{N}\sum_{j\le N_0}\delta_{x_j(0)}(x)\stackrel{w}{\Rightarrow} u_0(x)dx.
\end{align}
We introduce the sequence of space-time (normalized) empirical measures 
\begin{align}\label{emp-meas}
\xi^N\left(dt, dx\right):=\frac{1}{N}dt\sum_{j\le N(t)}\delta_{x_j(t)}(dx), \quad  N\in\N,
\end{align}
taking values in the space $\cM=\cM_T$ of nonnegative finite measures on $[0, T]\times\R^d$ endowed with the weak topology. Since these are random measures, the mappings
\begin{align*}
\omega\mapsto \xi^N(dt,dx,\omega)
\end{align*}
induce probability measures $\cP^N$ on the space $\cM$. That is, $\cP^N\in\cP(\cM)$.

We also introduce the \abbr{\abbr{F-KPP}} equation on $[0,T]\times\R^d$:
\begin{align}\label{fkpp}
\partial_t u= \frac{1}{2}\Delta u + u-u^2, \quad u(0,\cdot)=u_0,
\end{align}
with $u_0$ satisfying Condition \ref{kpp-init}.
We take as weak formulation of \eqref{fkpp} that for any $\phi\in C_c^{1,2}([0,T)\times\R^d)$,
\begin{align}\label{weak-informal}
0&=\int_{\R^d}u_0(x)\phi(0,x)dx+\int_0^T\int_{\R^d}u(t,x)\partial_t\phi(t,x)\, dxdt  \nonumber\\
&+\int_0^T\int_{\R^d}\left[\frac{1}{2}u(t,x)\Delta\phi(t,x)+(u-u^2)(t,x)\phi(t,x)\right]dxdt\,.
\end{align}
More precisely, see Definition \ref{weak-sol}.
The choice of this weak formulation, together with the uniqueness of the \abbr{F-KPP} equation is discussed in Section \ref{uniqueness}. Our main result is the following theorem, which is proved in Section \ref{sec:main}.

\begin{theorem}\label{thm-main}
Suppose that $\ep=\ep(N)$ is such that $\ep^{-d}\le CN$ for some finite constant $C$ and $\ep(N)\to0$ as $N\to\infty$. Then, for every finite $T$ and $d\ge 1$, the sequence of probability measures $\{\cP^N\}_N$ induced by $\{\omega\mapsto\xi^N(dt, dx, \omega)\}_N$ converges weakly in the space $\cP(\cM)$ to a Dirac measure on $\xi(dt, dx)\in\cM$. The measure $\xi$ is absolutely continuous with respect to the Lebesgues measure on $[0,T]\times\R^d$, i.e. $\xi(dt, dx)=u(t,x)dtdx$. The density $u(t,x)$ is the unique weak solution to the \abbr{\abbr{F-KPP}} equation \eqref{fkpp}, in the sense of \eqref{weak-informal}.
\end{theorem}

\section{Proof of the main result}\label{sec:main}
Our proof is based on adapting the strategy of \cite{HR}, which deals with scaling limits to coagulation-type \abbr{PDE}s. The key to the proof of the main result is an It\^o-Tanaka trick, well-known in the setting of \abbr{SDE}s. Specifically, for every $\ep$ and $T$, we define an auxiliary function $r^\ep(t,x)=r^{\ep, T}(t,x): [0,T]\times\R^d\to\R_+$, which is the unique solution to the \abbr{PDE} terminal value problem:
\begin{align}
\begin{cases}
\partial_t r^\ep(t,x) +\Delta r^\ep(t,x) + \theta^\ep(x)=0\\
r^\ep(T,x)=0
\end{cases}.\label{kolm}
\end{align}

Denoting by $C_0$ the maximum radius of the compact support of $\theta$, we have the following estimates for $r^\ep$ and $\nabla r^\ep$. 

\begin{proposition}\label{diff-est}
There exists finite constant $C(d, T, C_0)$ such that 
for any $x\in\R^d,\ep>0, t\in[0,T]$ we have that 
\begin{align}
&|r^\ep(t, x)|\le 
\begin{cases}\label{unif-bd}
Ce^{-C|x|^2}1_{\{|x|\ge 1\}}+C\left(|x|\vee \ep\right)^{2-d}1_{\{|x|<1\}}, \quad d\neq2 \\\\
Ce^{-C|x|^2}1_{\{|x|\ge 1\}}+C|\log\left(|x|\vee \ep\right)|1_{\{|x|<1\}}, \quad d=2
\end{cases}
\end{align}
\begin{align}
&|\nabla_x r^\ep(t,x)|\le Ce^{-C|x|^2}1_{\{|x|\ge 1\}}+C\left(|x|\vee \ep\right)^{1-d}1_{\{|x|<1\}}, \quad d\ge 1.\label{unif-bd-grad}
\end{align}
\end{proposition}
\begin{proof}
We first demonstrate \eqref{unif-bd}. Write%
\[
r^{\epsilon}\left(  t,x\right)  =u^{\epsilon}\left(  T-t,x\right)
\]%
with
\begin{align*}
\begin{cases}
\partial_{t}u^{\epsilon}\left(  t,x\right)     =\Delta u^{\epsilon}\left(
t,x\right)  +\theta^{\epsilon}\left(  x\right) \\
u^{\epsilon}\left(  0,x\right)    =0.
\end{cases}
\end{align*}
For each fixed $\ep>0$, the function $u^{\epsilon}\left(  t,x\right)  $ is of class $C^{1,2}([0,T]\times\R^d)$ (since we assumed that $\theta\in C^\alpha(\R^d)$ for some $\alpha\in(0,1)$) and it is given by the explicit formula%
\[
u^{\epsilon}\left(  t,x\right)  =\int_{0}^{t}\int_{\mathbb{R}^{d}}%
p_{t-s}\left(  x-y\right)  \theta^{\epsilon}\left(  y\right)  dyds
\]
where
\[
p_{t}\left(  x\right)  :=\left(  4\pi t\right)  ^{-d/2}\exp\left(
-\frac{\left\vert x\right\vert ^{2}}{4t}\right)  \qquad\text{for }t>0.
\]
We also have%
\begin{align*}
u^{\epsilon}\left(  t,x\right)   &  =\int_{0}^{t}\int_{\mathbb{R}^{d}}%
p_{s}\left(  y\right)  \theta^{\epsilon}\left(  x-y\right)  dyds\\
&  =\int_{\mathbb{R}^{d}}\theta^{\epsilon}\left(  x-y\right)  \left[  \int%
_{0}^{t}p_{s}\left(  y\right)  ds\right]  dy.
\end{align*}
This reformulation is crucial to understand the ``singularity" of $u^{\epsilon
}$ (let us repeat it is smooth, but it becomes singular at $x=0$ when
$\epsilon\rightarrow0$). Call%
\begin{align}\label{green}
K\left(  t,x\right)  :=\int_{0}^{t}p_{s}\left(  x\right)  ds
\end{align}
the kernel of this formula, such that
\[
u^{\epsilon}\left(  t,x\right)  =\int_{\mathbb{R}^{d}}\theta^{\epsilon}\left(
x-y\right)  K\left(  t,y\right)  dy.
\]
For $x\neq0$ the function $K\left(  t,x\right)  $ is well defined and smooth:
notice that $s\mapsto p_{s}\left(  x\right)  $ is integrable at $s=0$, and on
any set $\left[  0,T\right]  \times \ovl\cO  $ with $0\not\in\ovl\cO$, the function $p_{s}\left(  x\right)  $ is
uniformly continuous with all its derivatives in $x$ (extended equal to zero
for $t=0$). But for $x=0$ it is well defined only in dimension $d=1$.

We have, for $x\neq0$,
\begin{align*}
K\left(  t,x\right)   &  =\int_{0}^{t}\left(  4\pi s\right)  ^{-d/2}%
\exp\left(  -\frac{\left\vert x\right\vert ^{2}}{4s}\right)  ds\\
&  \stackrel{r=\frac{\left\vert x\right\vert ^{2}}{4s}}{=}\int_{\frac
{\left\vert x\right\vert ^{2}}{4t}}^{\infty}\left(  \frac{\pi\left\vert
x\right\vert ^{2}}{r}\right)  ^{-d/2}\exp\left(  -r\right)  \frac{\left\vert
x\right\vert ^{2}}{4r^{2}}dr\\
&  =\frac{1}{4\pi^{d/2}\left\vert x\right\vert ^{d-2}}\int_{\frac{\left\vert
x\right\vert ^{2}}{4t}}^{\infty}r^{\frac{d-4}{2}}\exp\left(  -r\right)  dr.
\end{align*}
Therefore%
\begin{align*}
K\left(  t,x\right)   &  =\frac{1}{\left\vert x\right\vert ^{d-2}}G\left(
t,x\right)  \\
G\left(  t,x\right)   &  :=\frac{1}{4\pi^{d/2}}\int_{\frac{\left\vert
x\right\vert ^{2}}{4t}}^{\infty}r^{\frac{d-4}{2}}\exp\left(  -r\right)  dr.
\end{align*}
Since
\begin{align*}
G\left(  t,x\right)   &  :=\frac{1}{4\pi^{d/2}}\int_{\frac{\left\vert
x\right\vert ^{2}}{4t}}^{\infty}r^{\frac{d-4}{2}}\exp\left(  -r\right)  dr\\
&  \leq\frac{1}{4\pi^{d/2}}\int_{\frac{\left\vert x\right\vert ^{2}}{4T}%
}^{\infty}r^{\frac{d-4}{2}}\exp\left(  -r\right)  dr\\
&  \overset{d\neq2}{\leq}A\exp\left(  -\alpha\left\vert x\right\vert
^{2}\right)  +B \left\vert x\right\vert   ^{d-2}1_{\{|x|<1\}}%
\end{align*}%
\[
G\left(  t,x\right)  \overset{d=2}{\leq}A\exp\left(  -\alpha\left\vert
x\right\vert ^{2}\right)  -B\log \left\vert x\right\vert 1_{\{|x|<1\}}
\]
for some $A,B,\alpha>0$. Therefore%
\begin{align*}
&  K\left(  t,x\right)  \overset{d\neq2}{\leq}\frac{1}{\left\vert x\right\vert
^{d-2}}\left[  A\exp\left(  -\alpha\left\vert x\right\vert ^{2}\right)
+B  \left\vert x\right\vert  ^{d-2}1_{\{|x|<1\}}\right]  \\
&  \leq\frac{1}{\left\vert x\right\vert ^{d-2}}A\exp\left(  -\alpha\left\vert
x\right\vert ^{2}\right)  +B1_{\{|x|<1\}}
\end{align*}%
\[
K\left(  t,x\right)  \overset{d=2}{\leq}A-B\log  \left\vert x\right\vert
1_{\{|x|<1\}} .
\]
It follows%
\begin{align*}
 u^{\epsilon}\left(  t,x\right)  &\overset{d\neq2}{\leq}A\int_{\mathbb{R}%
^{d}}\theta^{\epsilon}\left(  x-y\right)  \frac{1}{\left\vert y\right\vert
^{d-2}}\exp\left(  -\alpha\left\vert y\right\vert ^{2}\right)  dy+ B1_{\{|x|<1\}}\\
& \leq Ce^{-C|x|^2}1_{\{|x|\ge 1\}}+C\left(|x|\vee \ep\right)^{2-d}1_{\{|x|<1\}}
\end{align*}%
\[
u^{\epsilon}\left(  t,x\right)  \overset{d=2}{\leq}Ce^{-C|x|^2}1_{\{|x|\ge 1\}}+C|\log \left(|x|\vee \ep\right)|1_{\{|x|<1\}}.
\]
for some $C>0$.

To prove \eqref{unif-bd-grad}, we note that 
\begin{align*}
\nabla_x u^\ep(t,x)&=\int_0^t\int_{\R^d}\nabla p_s(y)\theta^\ep(x-y)dyds\\
&=\int_{\R^d}\theta^\ep(x-y)\left[\int_0^t\nabla p_s(y)ds\right]dy.
\end{align*}
Since 
\[
|\nabla p_t(x)|\le 2^{-1}(4\pi)^{-d/2}|x|t^{-\frac{d}{2}-1}\exp\left(-\frac{|x|^2}{4t}\right)
\]
we have that 
\begin{align*}
\left|\int_0^t\nabla p_s(y)ds\right|&\le  2^{-1}(4\pi)^{-d/2}\int_0^t|x|s^{-\frac{d}{2}-1}\exp\left(-\frac{|x|^2}{2}\right)ds\\
&\le 2^{-1}\pi^{-d/2}|x|^{1-d}\int_{\frac{|x|^2}{4T}}^\infty r^{\frac{d}{2}-1}e^{-r}dr\\
&\le |x|^{1-d}\left[A\exp(-\alpha|x|^2)+B|x|^d1_{\{|x|<1\}}\right]\\
&\le |x|^{1-d}A\exp(-\alpha|x|^2)+B1_{\{|x|<1\}}
\end{align*}
for some $A, B, \alpha>0$. Thus, we have
\begin{align*}
|\nabla_x u(t,x)|&\le A\int\theta^\ep(x-y)|y|^{1-d}\exp(-\alpha|y|^2)dy+B1_{\{|x|<1\}}\\
&\le Ce^{-C|x|^2}1_{\{|x|\ge 1\}}+C\left(|x|\vee \ep\right)^{1-d}1_{\{|x|<1\}}
\end{align*}
for some $C>0$. 
\end{proof}
\\\\
We need the following preliminary lemma.
\begin{lemma} \label{dm-theta}
For any $d\ge 1$ and finite $T$, there exists some finite $C=C(T, ||u_0||_{L^1})$ such that
\begin{align}
&\E\int_0^T\frac{1}{N^2}\sum_{j,k\le N(t)}\theta^\ep(x_j(t)-x_k(t))dt\le C.
\end{align}
\end{lemma}
\begin{proof}
By It\^o formula applied to the process $N(t)$ (the cardinality of alive particles) and taking expectation, the martingale vanishes and we get that
\begin{align*}
\E N(T)=\E N_0+\E\int_0^T\sum_{j\le N(t)}\left[1-\frac{1}{N}\sum_{k\le N(t)}\theta^\ep(x_j(t)-x_k(t))\right]dt
\end{align*}
implying that
\begin{align*}
\E\int_0^T\frac{1}{N}\sum_{j,k\le N(t)}\theta^\ep(x_j(t)-x_k(t))dt \le \E N_0+\E\int_0^TN(t)dt.
\end{align*}
The \abbr{RHS} is dominated from above by the same quantity calculated for a particle system with pure proliferation of unit rate (and no killing), and thereby is bounded by $e^TN\int u_0$.
\end{proof}
\\\\
We proceed to derive the limiting equation. Fixing any $\phi(t,x)\in C_c^{1,2}([0,T)\times \R^d)$, we consider the time dependent functional on $\eta$
\begin{align}
Q^N(t,\eta):=\frac{1}{N}\sum_{j\le i_{\text{max}}(\eta)}\phi(t, x_j).
\end{align}
By It\^o formula applied to the process $Q^N(t,\eta(t))$, we get that
\begin{align}\label{main-ito}
Q^N(T, \eta(T))-Q^N(0, \eta(0))=&\int_0^T \frac{1}{N}\sum_{j\le N(t)}\big(\partial_t+\frac{1}{2}\Delta_{x_j}\big)\phi(t,x_j(t))\, dt \nonumber\\
&+\int_0^T \frac{1}{N}\sum_{j\le N(t)}\Big[1-\frac{1}{N}\sum_{k\le N(t)}\theta^\ep(x_j(t)-x_k(t))\Big]\phi(t, x_j(t))\, dt +\wt M_T   \nonumber\\
=&\Big\langle \xi^N(dt, dx), \big(\partial_t+\frac{1}{2}\Delta+1\big)\phi(t,x)\Big\rangle   \nonumber\\
&- \int_0^T \frac{1}{N^2}\sum_{j,k\le N(t)}\theta^\ep(x_j(t)-x_k(t))\phi(t, x_j(t))\, dt +\wt M_T
\end{align}
where $\{\wt M_t\}$ is a martingale. We can readily control the martingale via its quadratic variation
\[
\mathbb E[\wt M_T^2]\le 4\int_0^T \mathbb E[A^{(1)}_t+A^{(2)}_t]\, dt 
\]
where
\begin{align*}
A^{(1)}_t:=&\frac{1}{N^2}\sum_{j\le N(t)}\left|\nabla_{x_j}\phi(t,x_j(t))\right|^2,\\
A^{(2)}_t:=&\frac{1}{N^2}\sum_{j\le N(t)}\Big[1+\frac{1}{N}\sum_{k\le N(t)}\theta^\ep((x_j(t)-x_k(t))\Big]\phi(t,x_j(t))^2.
\end{align*}
Since $\phi$ is a test function and $\E N(t)\le Ne^t\int u_0$, combined with Lemma \ref{dm-theta} we arrive at
\[
\E\int_0^T[A^{(1)}_t+A^{(2)}_t]\, dt \le \frac{C_{T,\phi}}{N}.
\]
Therefore, the martingale vanishes in $L^2(\P)$ (and in probability) in the limit $N\to\infty$. Further, since $\phi(T,\cdot)=0$, we have that $Q^N(T,\eta(T))=0$; whereas by our assumption on the initial condition, we have that $Q^N(0,\eta(0))\to\int\phi(0,x)u_0(x)dx$ in probability.
Regarding the last term of \eqref{main-ito}, i.e.
\begin{align}\label{key-term}
\int_0^T \frac{1}{N^2}\sum_{j,k\le N(t)}\theta^\ep\left(x_j(t)-x_k(t)\right)\phi\left(t, x_j(t)\right)\, dt,
\end{align}
we shall prove the following approximation in steps. To state it, let us denote 
\begin{align}\label{def:eta}
\eta^\delta(x):=\delta^{-d}\eta(\delta^{-1}x)
\end{align}
for a smooth, nonnegative, compactly supported function $\eta:\R^d\to\R_+$ with $\int \eta=1$. Fix also two smooth, compactly supported functions $\phi, \psi: \R^d\times[0, T)\to \R$. 
\begin{proposition}\label{product}
Suppose that $\ep=\ep(N)$ is such that $\ep^{-d}\le CN$ for some finite constant $C$ and $\ep(N)\to0$ as $N\to\infty$. Then, for any $d\ge 1$ and finite $T$, we have that
\begin{align*}
&\int_0^T\frac{1}{N^2}\sum_{j, k\le N(t)}\theta^\ep\left(x_j(t)-x_k(t)\right)\phi\left(t,x_j(t)\right)\psi\left(t,x_k(t)\right)dt\\
&=\int_0^T\int_{\R^d} \phi(t,w)\psi(t,w)\left(\xi^N*_x\eta^\delta\right)(t,w)^2dwdt+ Err(\epsilon, N, \delta)\,,
\end{align*}
for some error term that vanishes in the following limit
\[
\limsup_{\delta\to 0}\limsup_{N\to\infty}\; \E |Err(\epsilon,N, \delta)|=0, 
\]
and any $\eta:\R^d\to\R_+$ smooth, nonnegative, compactly supported with $\int \eta=1$. Here we used the shorthand
\[
\left(\xi^N*_x\eta^\delta\right)(t,w):=\frac{1}{N}\sum_{j\le N(t)}\eta^\delta(w-x_j(t)).
\]
\end{proposition}
{\textbf{Step I}}. Fixing $\ep, T$. Consider the time-dependent functional on $\eta$, indexed by $z\in\R^d$: 
\[
X^N_z(t, \eta):=\frac{1}{N^2}\sum_{j,k\le i_{\text{max}}(\eta)}r^\epsilon\left(t, x_j-x_k+z\right)\phi(t,x_j)\psi(t, x_k)
\]
where $r^\ep(t,x)$ is the auxiliary function defined in \eqref{kolm}. By It\^o formula applied to the process $(X^N_z-X^N_0)(t,\eta(t))$, we get that 
\begin{align*}
(X^N_z-X^N_0)(T, \eta(T))-(X^N_z-X^N_0)(0, \eta(0)) = \int_0^T \left((\partial_t+\wt\cL_N)(X_z-X_0)\right)(t, \eta(t))\, dt + M_T
\end{align*}
where $\{M_t\}$ is a martingale. Written out in detail, the \abbr{LHS} has one term 
\begin{align*}
H_0:=-\frac{1}{N^2}\sum_{j,k\le N_0}\big[r^\ep(0, x_j(0)-x_k(0)+z)-r^\ep(0, x_j(0)-x_k(0))\big]\phi(0,x_j(0))\psi(0,x_k(0))
\end{align*} 
and we have the following terms in the integrand of \abbr{RHS}
\begin{align*}
H_t(t)&:=\frac{1}{N^2}\sum_{j,k\le N(t)}\big[r^\epsilon(t,x_j(t)-x_k(t)+z)-r^\epsilon(t,x_j(t)-x_k(t))\big]\partial_t\Big(\phi(t,x_j(t))\psi(t, x_k(t))\Big).
\end{align*}
\begin{align*}
H_{xx}(t)&:=\frac{1}{N^2}\sum_{j,k\le N(t)}\Big[(\partial_t+\Delta) \big(r^\epsilon(t, x_j(t)-x_k(t)+z)-r^\epsilon(t, x_j(t)-x_k(t)\big)\Big]\phi(t,x_j(t))\psi(t, x_k(t))\\
&=\frac{1}{N^2}\sum_{j,k\le N(t)}\big[\theta^\epsilon(x_j(t)-x_k(t))-\theta^\epsilon(x_j(t)-x_k(t)+z)\big]\phi(t,x_j(t))\psi(t, x_k(t)).
\end{align*}
\begin{align*}
H_J(t)&:=\frac{1}{N^2}\sum_{j,k\le N(t)}\big[r^\epsilon(t,x_j(t)-x_k(t)+z)-r^\epsilon(t,x_j(t)-x_k(t))\big]\\
&\quad\quad\quad \cdot \frac{1}{2}\big(\Delta \phi(t,x_j(t))\psi(t, x_k(t))+\Delta\psi(t, x_k(t))\phi(t,x_j(t))\big).
\end{align*}
\begin{align*}
H_{x}(t)&:=\frac{1}{N^2}\sum_{j,k\le N(t)}\big[\nabla r^\ep(t,x_j(t)-x_k(t)+z)-\nabla r^\ep(t,x_j(t)-x_k(t))\big]\\
&\quad\quad\quad \cdot \frac{1}{2}\big(\nabla \phi(t,x_j(t))\psi(t,x_k(t))-\nabla \psi(t,x_k(t))\phi(t,x_j(t))\big).
\end{align*}
\begin{align*}
H_{C}(t):=&\frac{1}{N^2}\sum_{j,k\le N(t)}\Big[1-\frac{1}{N}\sum_{i\le N(t)}\theta^\ep\left(x_j(t)-x_i(t)\right)\Big]\\
&\quad\quad\quad \cdot\left[r^\epsilon\left(t,x_j(t)-x_k(t)+z\right)-r^\epsilon\left(t,x_j(t)-x_k(t)\right)\right]\phi(t,x_j(t))\psi(t, x_k(t))\\
&+\frac{1}{N^2}\sum_{j,k\le N(t)}\Big[1-\frac{1}{N}\sum_{i\le N(t)}\theta^\ep\left(x_k(t)-x_i(t)\right)\Big]\\
&\quad\quad\quad \cdot\left[r^\epsilon\left(t,x_j(t)-x_k(t)+z\right)-r^\epsilon\left(t,x_j(t)-x_k(t)\right)\right]\phi(t,x_j(t))\psi(t, x_k(t)).
\end{align*}
The martingale terms can be controlled via its quadratic variation
\begin{align*}
\E[M^2_T]\le 4\int_0^T\E[B^{(1)}_t+B^{(2)}_t]dt
\end{align*}
where 
\begin{align}\label{mg}
B^{(1)}_t:=&\frac{1}{N^4}\sum_{j\le N(t)}\left|\nabla_{x_j}\Big(\sum_{k\le N(t)}\big[r^\ep(t,x_j(t)-x_k(t)+z)- r^\ep(t,x_j(t)-x_k(t))\big]\phi(t,x_j(t))\psi(t,x_k(t))\Big)\right|^2 \nonumber\\
&+\frac{1}{N^4}\sum_{k\le N(t)}\left|\nabla_{x_k}\Big(\sum_{j\le N(t)}\big[r^\ep(t,x_j(t)-x_k(t)+z)- r^\ep(t,x_j(t)-x_k(t))\big]\phi(t,x_j(t))\psi(t,x_k(t))\Big)\right|^2.
\end{align}
\begin{align}\label{mg2}
B^{(2)}_t:=&\frac{1}{N^4}\sum_{j\le N(t)}\Big[1+\frac{1}{N}\sum_{i\le N(t)}\theta^\ep((x_j(t)-x_i(t))\Big]   \nonumber\\
&\quad\quad\quad \cdot\left|\sum_{k\le N(t)}\big[r^\ep(t,x_j(t)-x_k(t)+z)-r^\ep(t,x_j(t)-x_k(t))\big]\phi(t,x_j(t))\psi(t,x_k(t))\right|^2 \nonumber\\
&+\frac{1}{N^4}\sum_{k\le N(t)}\Big[1+\frac{1}{N}\sum_{i\le N(t)}\theta^\ep((x_k(t)-x_i(t))\Big]    \nonumber\\
&\quad\quad\quad \cdot\left|\sum_{j\le N(t)}\big[r^\ep(t,x_j(t)-x_k(t)+z)-r^\ep(t,x_j(t)-x_k(t))\big]\phi(t,x_j(t))\psi(t,x_k(t))\right|^2.
\end{align}

{\textbf{Step II}}. 
We show that among the previous terms, only $H_{xx}$ is significant, in a sense to be made precise. To this end, we need to bound the various other terms, of which there are significant similarities: one type of terms is a double sum involving the difference of $r^\ep$; the second type is a double sum involving the difference of $\nabla r^\ep$; and the third type is a triple sum involving the difference of $r^\ep$. 

We first prove a general proposition about a pure proliferation system that is naturally coupled to our system, from which some of our desired conclusions immediately follow.
\begin{proposition}\label{progenitor}
Let $d\ge 1$ and $\left(  x_{i}\left(  t\right)  \right)  $ be the pure proliferation model
with unit rate (no killing), with $N_0$ initial particles distributed
independently with density $(\int  u_0)^{-1}u_{0}$, for $u_{0}$ satisfying Condition \ref{kpp-init}. Let $T>0$ be
given and let $f(t,x): [0,T]\times\mathbb{R}^{d}%
\rightarrow\mathbb{R}_+$, $g(x):\R^d\to\R_+$ be bounded non-negative functions, with $g\in L^1(\R^d)$. Then there are constants $C_{T},c=c(d,T)>0$, independent of $N$ and $f, g$, such
that for any $t\in[0,T]$%
\begin{align}\label{general-double}
\mathbb{E}\left[  \sum_{i,j}f\left( t, x_{i}\left(  t\right)  -x_{j}\left(
t\right)  \right)  \right]  \leq N_02C^2_{T}\left\Vert f\right\Vert _{\infty
}\mathbb{+}N^{2}C_{T}^{2}\gamma^{2}e^{cR}\int_{\R^d}
f\left( t, x\right)  e^{-c\left\vert x\right\vert }dx,
\end{align}
and
\begin{align}
 &\mathbb{E}\left[ \sum_{i,j,k}f\left(t,  x_{i}\left(
t\right)  -x_{j}\left(  t\right)  \right)  g\left(x_{j}\left(  t\right)
-x_{k}\left(  t\right)  \right) \right]  \nonumber\\
 \leq& N_05C_T^3\left\Vert f\right\Vert _{\infty}\left\Vert g\right\Vert _{\infty}   \nonumber\\
& +N^{2}2C_T^3\gamma^{2}e^{cR}\left(  \left\Vert
f\right\Vert _{\infty}\int_{\R^d} g\left(  x\right)  e^{-c\left\vert x\right\vert
}dx+\left\Vert g\right\Vert _{\infty}\int_{\R^d} f\left(t,  x\right)  e^{-c\left\vert
x\right\vert }dx\right)  \text{ }  \nonumber\\
& +N^{3}C_T^3\gamma^{3}e^{cR}\left\Vert
g\right\Vert _{L^{1}}\int_{\R^d} f\left(t,  x\right)  e^{-c\left\vert x\right\vert }dx, \label{general-triple}
\end{align}
where the sum is extended to all particles alive at time $t$. The constant
$C_{T} (=e^T)$ is the average number of alive particles at time $T$, when starting from a single initial particle.
\end{proposition}

\begin{proof}
\textbf{Step 1}. Essential for the proof is the fact that the exponential
clocks of proliferation can be modeled a priori, therefore let us write a few
details in this direction for completeness. Particles, previously indexed by
$i$, will be indexed below by a multi-index $a$ of the form
\[
a=\left(  a_{1},...,a_{n}\right)
\]
with $n$ positive integer, $a_{1}\in\left\{  1,...,N_0\right\}  $,
$a_{2},...,a_{n}\in\left\{  1,2\right\}  $ (if $n\geq2$). Denote by
$\Lambda^{N}$ the set of all such multi-indexes. Given $a\in\Lambda^{N}$, we
denote by $n\left(  a\right)  $ the length of the string $a=\left(
a_{1},...,a_{n}\right)  $ defining $a$. We set
\[
a^{-1}=\left(  a_{1},...,a_{n-1}\right)
\]
when $n\geq2$. The heuristic idea behind these notations is that $a_{1}$
denotes the progenitor at time $t=0$; $a_{2},...,a_{n}$ describe the
subsequent story, where particle $a$ is a direct descendant of particle
$a^{-1}$. 

Each particle $a$ lives for a finite random time. On a probability space
$\left(  \Omega,\mathcal{F},\mathbb{P}\right)  $, assume to have a countable
family of independent Exponential r.v.'s $\tau^{a}$ of parameter $\lambda=1$,
indexed by $a\in\Lambda^{N}$. The time $\tau^{a}$ is the life span of particle
$a$; its interval of existence will be denoted by $[T_{0}^{a},T_{f}^{a})$ with
$T_{f}^{a}=T_{0}^{a}+\tau^{a}$. The random times $T_{0}^{a}$ are defined
recursively in $n\in\mathbb{N}$: if $n\left(  a\right)  =0$, $T_{0}^{a}=0$; if
$n\left(  a\right)  >0$,
\[
T_{0}^{a}=T_{0}^{a^{-1}}+\tau^{a^{-1}}=T_{f}^{a^{-1}}.
\]
We may now define the set of particles alive at time $t$: it is the set%
\[
\Lambda_{t}^{N}=\left\{  a\in\Lambda^{N}:t\in\lbrack T_{0}^{a},T_{f}%
^{a})\right\}  .
\]

Initial particles have a random initial position in the space $\mathbb{R}^{d}%
$: we assume that on the probability space $\left(  \Omega,\mathcal{F}%
,\mathbb{P}\right)  $ there are r.v.'s $X_{1},...,X_{N_0}$ distributed with
density $(\int u_{0})^{-1}u_0$ independent among themselves and with respect to the random
times $\tau^{a}$, $a\in\Lambda^{N}$.

Particles move as Brownian motions: we assume that on the probability space
$\left(  \Omega,\mathcal{F},\mathbb{P}\right)  $ there is a countable family
of independent Brownian motions $W^{a}$, $a\in\Lambda^{N}$, independent among
themselves and with respect to the random times $\tau^{a}$, $a\in\Lambda^{N}$
and the initial positions $X_{0}^{1},...,X_{0}^{N_0}$. The position $x_{t}^{a}$
of particle $a$ during its existence interval $[T_{0}^{a},T_{f}^{a})$ is
defined recursively in $n\in\mathbb{N}$ as follows: if $n\left(  a\right)
=0$, $x_{t}^{a}=X_{0}^{a}+W_{t}^{a}$ for $t\in\lbrack T_{0}^{a},T_{f}^{a})$;
if $n\left(  a\right)  >0$%
\[
x_{t}^{a}=x_{T_{0}^{a}}^{a^{-1}}+W_{t-T_{0}^{a}}^{a}\qquad\text{for }%
t\in\lbrack T_{0}^{a},T_{f}^{a}).
\]

\textbf{Step 2}. Given $k\in\left\{  1,...,N\right\}  $ and $a=\left(
a_{1},...,a_{n}\right)  \in\Lambda^{N}$, the process $x_{t}^{a}$ is formally
defined only for $t\in\lbrack T_{0}^{a},T_{f}^{a})$. Call $\widetilde{x}%
_{t}^{a}$ the related process, defined for all $t\geq0$ as follows: for each
$b=\left(  a_{1},...,a_{m}\right)  $ with $m\leq n$, on the interval
$[T_{0}^{b},T_{f}^{b})$ it is given by $x_{t}^{b}$; and on $[T_{f}^{a}%
,\infty)$ it is given by $x_{T_{0}^{a}}^{a^{-1}}+W_{t-T_{0}^{a}}^{a}$. The
process $\widetilde{x}_{t}^{a}$ is a Brownian motion with initial position
$X_{0}^{a_{1}}$. More precisely, if $\mathcal{G}$ denotes the $\sigma$-algebra
generated by the family $\left\{  \tau^{a};a\in\Lambda^{N}\right\}  $, then
the law of $\widetilde{x}_{t}^{a}$ conditioned to $\mathcal{G}$ is the law of
a Brownian motion with initial position $X_{0}^{a_{1}}$.

\textbf{Step 3}. With the notations of Step 1 above, we have to handle%
\[
\mathbb{E}\left[  \sum_{a,b\in\Lambda_{t}^{N}}f\left( t, x_{t}^{a}-x_{t}%
^{b}\right)  \right]  =\sum_{a,b\in\Lambda^{N}}\mathbb{E}\left[
1_{a\in\Lambda_{t}^{N}}1_{b\in\Lambda_{t}^{N}}f\left(t,  x_{t}^{a}-x_{t}%
^{b}\right)  \right]  .
\]
As explained in the previous step, let us denote the components of $a,b$ as
$a=\left(  a_{1},...a_{n}\right)  $, $b=\left(  b_{1},...b_{m}\right)  $, with
integers $n,m>0$, $a_{1},b_{1}\in\left\{  1,...,N_0\right\}  $ and all the other
entries in $\left\{  1,2\right\}  $. Then%
\begin{align*}
&  \sum_{a,b\in\Lambda^{N}}\mathbb{E}\left[  1_{a\in\Lambda_{t}^{N}}%
1_{b\in\Lambda_{t}^{N}}f\left(t,  x_{t}^{a}-x_{t}^{b}\right)  \right]  \\
&  =\sum_{a,b\in\Lambda^{N}:a_{1}=b_{1}}\mathbb{E}\left[  1_{a\in\Lambda
_{t}^{N}}1_{b\in\Lambda_{t}^{N}}f\left(t,  x_{t}^{a}-x_{t}^{b}\right)  \right]
+\sum_{a,b\in\Lambda^{N}:a_{1}\neq b_{1}}\mathbb{E}\left[  1_{a\in\Lambda
_{t}^{N}}1_{b\in\Lambda_{t}^{N}}f\left(t,  x_{t}^{a}-x_{t}^{b}\right)  \right]
.
\end{align*}
In the following computation, when we decompose a multi-index $a=\left(
a_{1},...a_{n}\right)  $ in the form $\left(  a_{1},a^{\prime}\right)  $ we
understand that $a^{\prime}$ does not exist in the case $n=1$, while
$a^{\prime}=\left(  a_{2},...a_{n}\right)  $ if $n\geq2$. We simply bound%
\begin{align*}
&  \sum_{a,b\in\Lambda^{N}:a_{1}=b_{1}}\mathbb{E}\left[  1_{a\in\Lambda
_{t}^{N}}1_{b\in\Lambda_{t}^{N}}f\left( t, x_{t}^{a}-x_{t}^{b}\right)  \right]
\\
&  =\sum_{a_{1}=1}^{N_0}\sum_{a^{\prime},b^{\prime}\in I}\mathbb{E}\left[
1_{\left(  a_{1},a^{\prime}\right)  \in\Lambda_{t}^{N}}1_{\left(
a_{1},b^{\prime}\right)  \in\Lambda_{t}^{N}}f\left( t, x_{t}^{\left(
a_{1},a^{\prime}\right)  }-x_{t}^{\left(  a_{1},b^{\prime}\right)  }\right)
\right]  \\
&  \leq\left\Vert f\right\Vert _{\infty}\sum_{a_{1}=1}^{N_0}\sum_{a^{\prime
},b^{\prime}\in I}\mathbb{E}\left[  1_{\left(  a_{1},a^{\prime}\right)
\in\Lambda_{t}^{N}}1_{\left(  a_{1},b^{\prime}\right)  \in\Lambda_{t}^{N}%
}\right]  \\
&  =N_0\left\Vert f\right\Vert _{\infty}\sum_{a^{\prime},b^{\prime}\in
I}\mathbb{E}\left[  1_{\left(  1,a^{\prime}\right)  \in\Lambda_{t}^{N}%
}1_{\left(  1,b^{\prime}\right)  \in\Lambda_{t}^{N}}\right]
\end{align*}
where $I$ denotes the set of binary sequences of finite length, and the last identity is due to the fact that the quantity $\mathbb{E}%
\left[  1_{\left(  a_{1},a^{\prime}\right)  \in\Lambda_{t}^{N}}1_{\left(
a_{1},b^{\prime}\right)  \in\Lambda_{t}^{N}}\right]  $ is independent of
$a_{1}$: then it is equal to%
\[
=N_0\left\Vert f\right\Vert _{\infty}\sum_{a^{\prime},b^{\prime}\in I}%
\mathbb{E}\left[  1_{\left(  1,a^{\prime}\right)  \in\Lambda_{t}^{1}%
}1_{\left(  1,b^{\prime}\right)  \in\Lambda_{t}^{1}}\right]
\]
where $\Lambda_{t}^{1}$ is the set of indexes relative to the case of a single
initial particle, and the identity holds because the presence of more initial
particles does not affect the expected values of the previous expression;
finally the previous quantity is equal to%
\begin{align*}
&  =N_0\left\Vert f\right\Vert _{\infty}\mathbb{E}\left[  \sum_{a,b\in
\Lambda_{t}^{1}}1\right]  \le N_0\left\Vert f\right\Vert _{\infty}\mathbb{E}\left[
\left\vert \Lambda_{t}^{1}\right\vert^2 \right]  \\
&  \leq N_0\left\Vert f\right\Vert _{\infty}\mathbb{E}\left[  \left\vert
\Lambda_{T}^{1}\right\vert^2 \right]\le N_0\|f\|_\infty(C^2_T+C_T),
\end{align*}
where we have denoted by $\left\vert \Lambda_{t}^{1}\right\vert $ the
cardinality of the set $\Lambda_{t}^{1}$, which is a Poisson random variable with finite mean $C_{T}=\mathbb{E}\left[
\left\vert \Lambda_{T}^{1}\right\vert \right]$, and we get
one addend of the inequality stated in the proposition.

Concerning the other sum,
\begin{align*}
&  \sum_{a,b\in\Lambda^{N}:a_{1}\neq b_{1}}\mathbb{E}\left[  1_{a\in
\Lambda_{t}^{N}}1_{b\in\Lambda_{t}^{N}}f\left(t,  x_{t}^{a}-x_{t}^{b}\right)
\right]  \\
&  =\sum_{a_{1}\neq b_{1}}\sum_{a^{\prime},b^{\prime}\in I}\mathbb{E}\left[
1_{\left(  a_{1},a^{\prime}\right)  \in\Lambda_{t}^{N}}1_{\left(
b_{1},b^{\prime}\right)  \in\Lambda_{t}^{N}}f\left( t, x_{t}^{\left(
a_{1},a^{\prime}\right)  }-x_{t}^{\left(  b_{1},b^{\prime}\right)  }\right)
\right]  \\
&  \leq N_0^{2}\sum_{a^{\prime},b^{\prime}\in I}\mathbb{E}\left[  1_{\left(
1,a^{\prime}\right)  \in\Lambda_{t}^{2}}1_{\left(  2,b^{\prime}\right)
\in\Lambda_{t}^{2}}f\left(t,  x_{t}^{\left(  1,a^{\prime}\right)  }%
-x_{t}^{\left(  2,b^{\prime}\right)  }\right)  \right]
\end{align*}
where the last inequality, involving a system with only two initial particles,
can be explained similarly to what done above. Recalling the notation of Step
2 above, the previous expression is equal to%
\[
=N_0^{2}\sum_{a^{\prime},b^{\prime}\in I}\mathbb{E}\left[  1_{\left(
1,a^{\prime}\right)  \in\Lambda_{t}^{2}}1_{\left(  2,b^{\prime}\right)
\in\Lambda_{t}^{2}}f\left(t,  \widetilde{x}_{t}^{\left(  1,a^{\prime}\right)
}-\widetilde{x}_{t}^{\left(  2,b^{\prime}\right)  }\right)  \right]  .
\]
Now we use the fact that the laws of processes indexed by 1 and 2 are
independent and the law of $\widetilde{x}_{t}^{\left(  1,a^{\prime}\right)  }$
conditioned to $\mathcal{G}^{1}$ is a Brownian motion with initial position
$X_{0}^{1}$, where $\mathcal{G}^{1}$ is the $\sigma$-algebra generated by the
family $\left\{  \tau^{\left(  1,a^{\prime}\right)  };\left(  1,a^{\prime
}\right)  \in\Lambda^{1}\right\}  $; and similarly for $\widetilde{x}%
_{t}^{\left(  2,b^{\prime}\right)  }$ with respect to $\mathcal{G}^{2}$,
similarly defined. Thus, after taking conditional expectation with respect to
$\mathcal{G}^{1}\vee\mathcal{G}^{2}$ inside the previous expected value, we
get that the previous expression is equal to
\[
=N_0^{2}\sum_{a^{\prime},b^{\prime}\in I}\mathbb{P}\left(  \left(  1,a^{\prime
}\right)  \in\Lambda_{t}^{2}\right)  \mathbb{P}\left(  \left(  2,b^{\prime
}\right)  \in\Lambda_{t}^{2}\right)  \mathbb{E}\left[  f\left( t, W_{t}%
^{1}-W_{t}^{2}+X_{0}^{1}-X_{0}^{2}\right)  \right]
\]
where $W_{t}^{i}$, $i=1,2$ are two independent Brownian motions in
$\mathbb{R}^{d}$, independent also of $X_{0}^{1},X_{0}^{2}$. We may simplify
the previous expression to
\[
=N_0^{2}\left(  \sum_{a^{\prime}\in I}\mathbb{P}\left(  \left(  1,a^{\prime
}\right)  \in\Lambda_{t}^{1}\right)  \right)  ^{2}\mathbb{E}\left[  f\left(t,
\sqrt{2}W_{t}+X_{0}^{1}-X_{0}^{2}\right)  \right]
\]
where $W_{t}$ is a Brownian motion in $\mathbb{R}^{d}$ independent of
$X_{0}^{1},X_{0}^{2}$. One has
\begin{align*}
\sum_{a^{\prime}\in I}\mathbb{P}\left(  \left(  1,a^{\prime}\right)
\in\Lambda_{t}^{1}\right)    & =\mathbb{E}\left[  \sum_{a\in\Lambda_{t}^{1}%
}1\right]  =\mathbb{E}\left[  \left\vert \Lambda_{t}^{1}\right\vert \right]
\\
& \leq\mathbb{E}\left[  \left\vert \Lambda_{T}^{1}\right\vert \right]  =C_{T}.
\end{align*}
Moreover, denoting $\overline{u_{0}}\left(  x\right)  =u_{0}\left(  -x\right)
$,
\begin{align*}
&\mathbb{E}\left[  f\left( t, \sqrt{2}W_{t}+X_{0}^{1}-X_{0}^{2}\right)  \right]
 =\|u_0\|_{L^1}^{-2}\int\mathbb{E}\left[  f\left( t, \sqrt{2}W_{t}+x\right)  \right]  \left(
\overline{u_{0}}\ast u_{0}\right)  \left(  x\right)  dx\\
& =\|u_0\|_{L^1}^{-2}\left\langle e^{t\Delta}f(t,\cdot),\overline{u_{0}}\ast u_{0}\right\rangle
=\|u_0\|_{L^1}^{-2}\left\langle f(t,\cdot),e^{t\Delta}\left(  \overline{u_{0}}\ast u_{0}\right)
\right\rangle \\
& =\|u_0\|_{L^1}^{-2}\int f\left( t, x\right)  \mathbb{E}\left[  \left(  \overline{u_{0}}\ast
u_{0}\right)  \left(  \sqrt{2}W_{t}+x\right)  \right]  dx.
\end{align*}
Now we may estimate%
\begin{align*}
\mathbb{E}\left[  \left(  \overline{u_{0}}\ast u_{0}\right)  \left(  \sqrt
{2}W_{t}+x\right)  \right]    & \leq\gamma%
^{2}\mathbb{E}\left[  1_{B\left(  0,R\right)  }\left(  \sqrt{2}W_{t}+x\right)
\right]  \\
& =\gamma^{2}\mathbb{P}\left(  \sqrt{2}W_{t}\in
B\left(  x,R\right)  \right)  \\
& \leq\gamma^{2}\mathbb{P}\left(  \left\vert
\sqrt{2}W_{t}\right\vert \geq\left\vert x\right\vert -R\right)  \\
& \leq\gamma^{2}e^{cR}e^{-c\left\vert
x\right\vert }%
\end{align*}
for some constant $c=c(d,T)>0$. Therefore, summarizing,
\begin{align*}
& N_0^{2}\left(  \sum_{a^{\prime}\in I}\mathbb{P}\left(  \left(  1,a^{\prime
}\right)  \in\Lambda_{t}^{1}\right)  \right)  ^{2}\mathbb{E}\left[  f\left(t,
\sqrt{2}W_{t}+X_{0}^{1}-X_{0}^{2}\right)  \right]  \\
& \leq N_0^{2}C_{T}^{2}\|u_0\|_{L^1}^{-2}\gamma^{2}e^{cR}\int
f\left( t, x\right)  e^{-c\left\vert x\right\vert }dx.
\end{align*}
This completes the proof of \eqref{general-double}.

\textbf{Step 4}. Now we turn to demonstrate \eqref{general-triple}, i.e.
\begin{align*}
&\E\left[\sum_{a,b,c\in\Lambda_t^N}f(t,x_t^a-x_t^b)g(x_t^a-x_t^c)\right]=\sum_{a,b,c\in\Lambda^N}\E\left[1_{a\in\Lambda_t^N}1_{b\in\Lambda_t^N}1_{c\in\Lambda_t^N}f(t,x_t^a-x_t^b)g(x_t^a-x_t^c)\right].
\end{align*}
We devide the above sum into five cases:
\begin{align*}
&\sum_{a,b,c\in\Lambda^N: a_1=b_1=c_1}+\sum_{a,b,c\in\Lambda^N: a_1=b_1\neq c_1}+\sum_{a,b,c\in\Lambda^N: a_1=c_1\neq b_1}+\sum_{a,b,c\in\Lambda^N: b_1=c_1\neq a_1}+\sum_{a,b,c\in\Lambda^N: a_1\neq b_1, a_1\neq c_1, b_1\neq c_1}\\
&:=S_1+S_2+S_3+S_4+S_5.
\end{align*}
Firstly, 
\begin{align*}
S_1&\le \|f\|_\infty\|g\|_\infty\sum_{a_1=1}^{N_0}\sum_{a',b',c'\in I}\E\left[1_{(a_1,a')\in \Lambda_t^N}1_{(a_1,b')\in\Lambda_t^N}1_{(a_1,c')\in\Lambda_t^N}\right]\\
&\le N_0\|f\|_\infty\|g\|_\infty\E\left[\left|\Lambda_t^1\right|^3\right]\le 
N_0\|f\|_\infty\|g\|_\infty5C_T^3.
\end{align*}
Secondly, 
\begin{align*}
S_2&\le \|f\|_\infty\sum_{a,b,c\in\Lambda^N: a_1=b_1\neq c_1}\E\left[1_{a,b,c\in\Lambda_t^N}g\left(x_t^a-x_t^c\right)\right]\\
&=\|f\|_\infty\sum_{a_1\neq c_1}\sum_{a',b',c'\in I}\E\left[1_{(a_1,a')\in\Lambda_t^N}1_{(a_1,b')\in\Lambda_t^N}1_{(c_1,c')\in\Lambda_t^N}g\left(x_t^{(a_1,a')}-x_t^{(c_1,c')}\right)\right]\\
&\le \|f\|_\infty N_0^2\sum_{a',b',c'\in I}\E\left[1_{(1,a')\in\Lambda_t^1}1_{(1,b')\in\Lambda_t^1}1_{(2,c')\in\Lambda_t^2}g\left(x_t^{(1,a')}-x_t^{(2,c')}\right)\right]\\
&= \|f\|_\infty N_0^2\sum_{a',b',c'\in I}\E\left[1_{(1,a')\in\Lambda_t^1}1_{(1,b')\in\Lambda_t^1}1_{(2,c')\in\Lambda_t^2}g\left(\wt x_t^{(1,a')}- \wt x_t^{(2,c')}\right)\right].
\end{align*}
Noting that $\wt x_t^{(1,\cdot)}, \wt x_t^{(2,\cdot)}$ are independent processes
\begin{align*}
&=\|f\|_\infty N_0^2\sum_{a',b',c'\in I}\P\left((1,a'), (1,b')\in\Lambda_t^1\right)\P\left((2,c')\in\Lambda_t^2\right)\E\left[g\left(W_t^1-W_t^2+X_0^1-X_0^2\right)\right]\\
&\le \|f\|_\infty N_0^2\E\left[\left|\Lambda_t^1\right|^2\right]\E\left[\left|\Lambda_t^2\right|\right]\E\left[g\left(W_t^1-W_t^2+X_0^1-X_0^2\right)\right]
\end{align*}
for two independent auxiliary Brownian motions $W_t^1, W_t^2$. Similarly to already analyzed in Step 3, it is bounded by 
\begin{align*}
\le\|f\|_\infty N_0^22C_T^3\|u_0\|_{L^1}^{-2}\gamma^2e^{cR}\int g(x)e^{-c|x|}dx.
\end{align*}
Thirdly, 
\begin{align*}
S_3&\le \|g\|_\infty\sum_{a,b,c\in\Lambda^N: a_1=c_1\neq b_1}\E\left[1_{a,b,c\in\Lambda_t^N}f\left(t,x_t^a-x_t^b\right)\right]\\
&=\|g\|_\infty\sum_{a_1\neq b_1}\sum_{a',b',c'\in I}\E\left[1_{(a_1,a')\in\Lambda_t^N}1_{(b_1,b')\in \Lambda_t^N}1_{(a_1,c')\in\Lambda_t^N}f\left(x_t^{(a_1,a')}-x_t^{(b_1,b')}\right)\right]\\
&\le \|g\|_\infty N_0^2\sum_{a',b',c'\in I}\E\left[1_{(1,a')\in\Lambda_t^1}1_{(2,b')\in \Lambda_t^2}1_{(1,c')\in\Lambda_t^1}f\left(x_t^{(1,a')}-x_t^{(2,b')}\right)\right]\\
&=\|g\|_\infty N_0^2\sum_{a',b',c'\in I}\P\left((1,a'),(1,c')\in\Lambda_t^1\right)\P\left((2,b')\in\Lambda_t^2\right)\E\left[f\left(W_t^1-W_t^2+X_0^1-X_0^2\right)\right]\\
&=\|g\|_\infty N_0^2\E\left[\left|\Lambda_t^1\right|^2\right]\E\left[\left|\Lambda_t^2\right|\right]\E\left[f\left(W_t^1-W_t^2+X_0^1-X_0^2\right)\right]\\
&\le\|g\|_\infty N_0^22C_T^3\|u_0\|_{L^1}^{-2}\gamma^2e^{cR}\int f(t,x)e^{-c|x|}dx.
\end{align*}
The analysis of $S_4$ is analogous to $S_3$, and finally, 
\begin{align*}
S_5&=\sum_{a,b,c\in \Lambda^N: a_1\neq b_1 \neq c_1}\E\left[1_{a,b,c\in\Lambda_t^N}f\left(t,x_t^a-x_t^b\right)g\left(x_t^a-x_t^c\right)\right]\\
&\le N_0^3\sum_{a',b',c'\in I}\E\left[1_{(1,a')\in \Lambda_t^1}1_{(2,b')\in\Lambda_t^2}1_{(3,c')\in \Lambda_t^3}f\left(t,x_t^{(1,a')}-x_t^{(2,b')}\right)g\left(x_t^{(1,a')}-x_t^{(3,c')}\right)\right]\\
&= N_0^3\sum_{a',b',c'\in I}\E\left[1_{(1,a')\in \Lambda_t^1}1_{(2,b')\in\Lambda_t^2}1_{(3,c')\in \Lambda_t^3}f\left(t,\wt x_t^{(1,a')}-\wt x_t^{(2,b')}\right)g\left(\wt x_t^{(1,a')}-\wt x_t^{(3,c')}\right)\right]
\end{align*}
Noting that $\wt x_t^{(1,\cdot)}, \wt x_t^{(2, \cdot)}, \wt x_t^{(3, \cdot)}$ are independent processes, 
\begin{align*}
=N_0^3\sum_{a',b',c'\in I}\P\left((1,a')\in\Lambda_t^1\right)&\P\left((2,b')\in\Lambda_t^2\right)\P\left((3,c')\in\Lambda_t^3\right)\\
&\cdot\E\left[f\left(t,W_t^1-W_t^2+X_0^1-X_0^2\right)g\left(W_t^1-W_t^3+X_0^1-X_0^3\right)\right]\\
=N_0^3\E\left[\left|\Lambda_t^1\right|\right]\E\left[\left|\Lambda_t^2\right|\right]\E\left[\left|\Lambda_t^3\right|\right]&
\E\left[f\left(t,W_t^1-W_t^2+X_0^1-X_0^2\right)g\left(W_t^1-W_t^3+X_0^1-X_0^3\right)\right],
\end{align*}
for three independent auxiliary Brownian motions $W_t^1, W_t^2, W_t^3$. Conditioning on $W_t^1, W_t^2, X_0^1, X_0^2$, and we compute
\begin{align*}
&\E\left[g\left(W_t^1-W_t^3+X_0^1-X_0^3\right)\; \Big|\; W_t^1, W_t^2, X_0^1, X_0^2\right]\\
&=\|u_0\|_{L^1}^{-1}\E\left[\int g\left(W_t^1+X_0^1-x\right)\left(e^{\frac{1}{2}\Delta}u_0\right)(x)dx\; \Big|\; W_t^1, W_t^2, X_0^1, X_0^2\right]\\
&\le \|u_0\|_{L^1}^{-1}\gamma\|g\|_{L^1}.
\end{align*}
Thus, we obtain that 
\begin{align*}
S_5&\le N_0^3C_T^3\|u_0\|_{L^1}^{-1}\gamma\|g\|_{L^1}\E\left[f\left(t,W_t^1-W_t^2+X_0^1-X_0^2\right)\right]\\
&\le N_0^3C_T^3\|u_0\|_{L^1}^{-3}\gamma^3\|g\|_{L^1}e^{cR}\int f(t,x)e^{-c|x|}dx.
\end{align*}
This completes the proof of \eqref{general-triple}.
\end{proof}

\begin{corollary}\label{cor:double}
Let $d\ge 1$ and $\ep=\ep(N)$ as in the statement of the main theorem. For any $T$ finite and $\phi, \psi\in C_c^\infty([0,T)\times\R^d)$, we have that 
\begin{align}
&\limsup_{|z|\to0}\limsup_{N\to\infty}\nonumber\\
&\quad\E\int_0^T\frac{1}{N^2}\sum_{j,k\le N(t)}\left|r^\ep(t,x_j(t)-x_k(t)+z)-r^\ep(t,x_j(t)-x_k(t))\right||\phi|(t,x_j(t))|\psi|(t,x_k(t))dt=0. \label{double-sum-1}\\
&\limsup_{|z|\to0}\limsup_{N\to\infty}\nonumber\\
&\quad\E\int_0^T\frac{1}{N^2}\sum_{j,k\le N(t)}\left|\nabla r^\ep(t,x_j(t)-x_k(t)+z)-\nabla r^\ep(t,x_j(t)-x_k(t))\right||\phi|(t,x_j(t))|\psi|(t,x_k(t))dt=0.\label{double-sum-2}\\
&\limsup_{|z|\to0}\limsup_{N\to\infty}  \nonumber\\
&\E\int_0^T\frac{1}{N^3}\sum_{i,j,k\le N(t)}\theta^\ep\left(x_j(t)-x_k(t)\right)\left|r^\epsilon(t, x_j(t)-x_i(t)+z)-r^\epsilon(t, x_j(t)-x_i(t))\right||\phi|(t,x_i(t))|\psi|(t, x_j(t))dt =0.\label{triple-sum}
\end{align}
\end{corollary}

\begin{proof}
Note that our particle system can be coupled with a system of pure proliferation of unit rate (with no killing), so that the former is a strict subset of the latter. Hence, it is an upper bound to compute \eqref{double-sum-1}-\eqref{triple-sum} for the pure proliferation process. We proceed to do so in the rest of the proof, while abusing notations, still using the letter $x_j(t)$ to denote particle positions (now for a different system), and $N(t)$ the cardinality of particles.

Firstly, by \eqref{general-double} applied to the function $f(t,x):=|r^\ep(t,x+z)-r^\ep(t,x)|$, and by \eqref{unif-bd}, upon bounding $\phi, \psi$ by constants, we get that
\begin{align*}
&C_{\phi,\psi}\E\int_0^T\frac{1}{N^2}\sum_{j,k\le N(t)}\left|r^\ep(t,x_j(t)-x_k(t)+z)-r^\ep(t,x_j(t)-x_k(t))\right|dt\\
&\le C_{\phi,\psi,T, u_0}\frac{1}{N}\left\Vert r^{\epsilon}\left( \cdot, \cdot+z\right)
-r^{\epsilon}\left( \cdot, \cdot\right)  \right\Vert _{\infty}\\
&\quad\quad\quad +C_{\phi,\psi,T,u_0}e^{cR}\left( \int_0^T \int\left\vert r^{\epsilon}\left(t,  x+z\right)
-r^{\epsilon}\left(t,  x\right)  \right\vert e^{-c\left\vert x\right\vert
}dxdt\right)\\
&\leq 
\begin{cases}
C_{\phi,\psi, T, u_0}^{\prime}\left(\frac{\epsilon^{2-d}}{N}+\int_0^T\int\left\vert r^{\epsilon}\left( t, x+z\right)  -r^{\epsilon}\left(t,
x\right)  \right\vert e^{-c\left\vert x\right\vert }dxdt\right), \quad d\neq 2,\\\\
C_{\phi,\psi,T, u_0}^{\prime}\left(\frac{|\log\ep|}{N}+\int_0^T\int\left\vert r^{\epsilon}\left( t, x+z\right)  -r^{\epsilon}\left(t,
x\right)  \right\vert e^{-c\left\vert x\right\vert }dxdt\right), \quad d=2.
\end{cases}
\end{align*}
The first term is negligible for our range of $\ep(N)$. 
For the second term, for the kernel $K$ defined at \eqref{green},
\begin{align*}
& \int_0^T\int\left\vert r^{\epsilon}\left( t, x+z\right)  -r^{\epsilon}\left(t,
x\right)  \right\vert e^{-c\left\vert x\right\vert }dxdt\\
& =\int_0^T\int\left\vert \left(  \theta^{\epsilon}\ast\left(  K\left( t, \cdot
+z\right)  -K\left(t,  \cdot\right)  \right)  \right)  \left(  x\right)
\right\vert e^{-c\left\vert x\right\vert }dxdt\\
& \leq\int_0^T\int\left(  \theta^{\epsilon}\ast\left\vert K\left( t, \cdot+z\right)
-K\left( t, \cdot\right)  \right\vert \right)  \left(  x\right)  e^{-c\left\vert
x\right\vert }dxdt\\
& = \int_0^T \int\left\vert K\left(t,  x+z\right)  -K\left( t, x\right)  \right\vert \left(
\theta^{\epsilon}\ast e^{-c\left\vert \cdot\right\vert }\right)  \left(
x\right)  dxdt.
\end{align*}
Now we take the two limits; as $N\to\infty$ hence $\ep=\epsilon(N)\rightarrow0$, by Lebesgue dominated
convergence theorem we get
\[
\rightarrow\int_0^T\int\left\vert K\left( t, x+z\right)  -K\left(t,  x\right)  \right\vert
e^{-c\left\vert x\right\vert }dxdt.
\]
Then, as $\left\vert z\right\vert \rightarrow0$, again by Lebesgue dominated
convergence theorem we get that the limit is zero.

Next, the proof of \eqref{double-sum-2} is similar, and only involves a minor change. We apply \eqref{general-double} with the new function $f(t, x):=|\nabla r^\ep(t, x+z)-\nabla r^\ep(t, x)|$, and by \eqref{unif-bd-grad} we obtain for all $d\ge 1$,
\begin{align*}
&C_{\phi,\psi}\E\int_0^T\frac{1}{N^2}\sum_{j,k\le N(t)}\left|\nabla r^\ep(t,x_j(t)-x_k(t)+z)-\nabla r^\ep(t,x_j(t)-x_k(t))\right|\\
&\leq 
C_{\phi,\psi,T, u_0}^{\prime}\left(\frac{\epsilon^{1-d}}{N}+\int_0^T\int\left\vert \nabla r^{\epsilon}\left( t, x+z\right)  -\nabla r^{\epsilon}\left(t,
x\right)  \right\vert e^{-c\left\vert x\right\vert }dxdt\right).
\end{align*}
Then, we have that 
\begin{align*}
&\int_0^T\int|\nabla r^\ep(t,x+z)-\nabla r^\ep(t,x)|e^{-c|x|}dxdt\\
&\le \int_0^T\int|\nabla K(t,x+z)-\nabla K(t,x)|\left(\theta^\ep*e^{-c|\cdot|}\right)(x)dxdt
\end{align*}
still converges to zero as $N\to\infty$ followed by $|z|\to0$ by the dominated convergence theorem. Indeed, $|\nabla K(t,x)|$ has a singularity of order $|x|^{1-d}$ near $0$ hence integrable for all $d$.

Lastly, turning to \eqref{triple-sum}. By \eqref{general-triple}, applied to the functions $f(t,x):=|r^\ep(t,x+z)-r^\ep(t,x)|$, $g(x)=\theta^\ep(x)$, and by \eqref{unif-bd}, upon bounding $\phi, \psi$ by constants, we get that 
\begin{align*}
&C_{\phi,\psi}\E\int_0^T\frac{1}{N^3}\sum_{i,j,k\le N(t)}\theta^\ep\left(x_j(t)-x_k(t)\right)\left|r^\epsilon(t, x_j(t)-x_i(t)+z)-r^\epsilon(t, x_j(t)-x_i(t))\right|dt\\
&\le C_{\phi,\psi,T,u_0}\frac{1}{N^2}\ep^{2-2d}\|\theta\|_\infty\\
&+C_{\phi,\psi,T,u_0}\frac{1}{N}\left(\ep^{2-d}\int \theta^\ep(x)e^{-|x|}dx+\ep^{-d}\int_0^T\int |r^\ep(t,x+z)-r^\ep(t,x)|e^{-c|x|}dxdt\right)\\
&+C_{\phi,\psi,T,u_0}\|\theta^\ep\|_{L^1}\int_0^T\int |r^\ep(t,x+z)-r^\ep(t,x)|e^{-c|x|}dxdt\\
&\le C'_{\phi,\psi,T,u_0}\ep^2(\|\theta\|_\infty+\|\theta^\ep\|_{L^1})+C'_{\phi,\psi,T,u_0}\int_0^T\int |r^\ep(t,x+z)-r^\ep(t,x)|e^{-c|x|}dxdt
\end{align*}
if $d\neq 2$, and when $d=2$ there is a $|\log\ep|$ correction, 
where we also used the relation $\ep(N)^{-d}\le CN$.
Since the first term is negligible in $\ep$, and the second term is already analyzed in \eqref{double-sum-1}, converging to zero as $N\to\infty$ followed by $|z|\to 0$, the lemma is proved.
\end{proof}

\begin{remark}
Though Corollary \ref{cor:double} does not give a rate of convergence for the quantities involved, via a different proof we can have quantitative estimates that may be of independent interest: there exists some finite constant $C=C(T,d,C_0,R, \gamma)$ such that for any $0<\ep\le |z|$ small enough, we have
\begin{align}
&\E\int_0^T\frac{1}{N^2}\sum_{j,k\le N(t)}\left|r^\ep(t,x_j(t)-x_k(t)+z)-r^\ep(t,x_j(t)-x_k(t))\right||\phi|(t,x_j(t))|\psi|(t,x_k(t))dt  \nonumber\\
&\le 
\begin{cases}
C\left(|z|^{\frac{2}{d+1}}+\frac{\ep^{2-d}}{N}\right), \quad d\neq 2,\\\\
C\left(|z|^{\frac{2}{3}}+\frac{|\log\ep|}{N}\right), \quad d=2
\end{cases}
\end{align}
and
\begin{align}
&\E\int_0^T\frac{1}{N^2}\sum_{j,k\le N(t)}\left|\nabla r^\ep(t,x_j(t)-x_k(t)+z)-\nabla r^\ep(t,x_j(t)-x_k(t))\right||\phi|(t,x_j(t))|\psi|(t,x_k(t))dt  \nonumber\\
&\le C\left(|z|^{\frac{1}{d+1}}+\frac{\ep^{1-d}}{N}\right), \quad d\ge1
\end{align}
and 
\begin{align}
&\E\int_0^T\frac{1}{N^3}\sum_{i,j,k\le N(t)}\theta^\ep(x_j(t)-x_k(t))\left|r^\epsilon(t, x_j(t)-x_i(t)+z)-r^\epsilon(t, x_j(t)-x_i(t))\right||\phi|(t,x_i(t))|\psi|(t, x_j(t))dt  \nonumber\\
&\le
\begin{cases}
C\left(|z|^{\frac{2}{d+1}}+\frac{\ep^{-d}}{N}|z|^{\frac{2}{d+1}}+\frac{\ep^{2-d}}{N}\right), \quad d\neq 2,\\\\
C\left(|z|^{\frac{2}{d+1}}+\frac{\ep^{-d}}{N}|z|^{\frac{2}{d+1}}+\frac{|\log\ep|}{N}\right), \quad d=2.
\end{cases}
\end{align}
\end{remark}

\begin{remark}\label{rmk:mg}
When we proceed to bound the martingale terms $B^{(1)}$, $B^{(2)}$ \eqref{mg}-\eqref{mg2}, we are faced with a minor problem not present in Corollary \ref{cor:double}, namely, after applying the elementary inequality $(\sum_{i=1}^n a_i)^2\le n\sum_{i=1}^na_i^2$, we have sums of square terms $|r^\ep|^2$ or $|\nabla r^\ep|^2$. This can be dealt with, by bounding one of $|r^\ep|$ (resp. $|\nabla r^\ep|$) in the square crudely by $C\ep^{2-d}$ ($d\neq 2$) or $C|\log\ep|$ ($d=2$) (resp. $C\ep^{1-d}$), and leave with the other one, for which we are back to one of the three statements of Corollary \ref{cor:double}. Also note that we are saved by the prefactor $N^{-4}$ in this case.
\end{remark}
{\bf{Step III}}.
Given Proposition \ref{progenitor} and Corollary \ref{cor:double}, we can proceed to finish the proof of Proposition \ref{product}, as in \cite[page 42-43]{HR}. By applying this corollary, together with Remark \ref{rmk:mg}, we see that the terms coming out of the application of It\^o-Tanack trick, namely $H_0, H_t, H_J, H_x, H_C, B^{(1)}, B^{(2)}$, all vanish in the limit as $N\to\infty$ followed by $|z|\to 0$ (where $\ep=\ep(N)\to 0$ is such that $\ep^{-d}\le CN$). The only outstanding term is $H_{xx}$, whereby we get that 
\begin{align}\label{boltzmann}
&\int_0^T\frac{1}{N^2}\sum_{j,k\le N(t)}\theta^\ep(x_j(t)-x_k(t))\phi(t,x_j(t))\psi(t,x_k(t))dt   \nonumber\\
&=\int_0^T\frac{1}{N^2}\sum_{j,k\le N(t)}\theta^\ep(x_j(t)-x_k(t)+z)\phi(t,x_j(t))\psi(t,x_k(t))dt+Err(\ep,N, |z|)
\end{align}
where $Err(\ep,N, |z|)$ vanishes in the following limit
\begin{align*}
\limsup_{|z|\to0}\limsup_{N\to\infty}\E\left|Err(\ep,N,|z|)\right|=0.
\end{align*}
Since the \abbr{LHS} of \eqref{boltzmann} is independent of $z$, take any nonnegative smooth and compactly supported function $\eta: \R^d\to\R_+$ with $\int_{\R^d}\eta=1$, we have 
\begin{align}\label{artificial}
&\int_0^T\frac{1}{N^2} \sum_{j,k\le N(t)}\theta^\ep\left(x_j(t)-x_k(t)\right)\phi(t,x_j(t))\psi(t,x_k(t))\nonumber\\
&= \int_0^T\iint_{\R^{2d}} \frac{1}{N^2}\sum_{j,k}\theta^\ep\left(x_j(t)-x_k(t)-z_1+z_2\right)\eta^\delta(z_1)\eta^\delta(z_2)\phi(t,x_j(t))\psi(t,x_k(t))dz_1dz_2dt+ Err(\ep, N,\delta)
\end{align}
with
\[
\limsup_{\delta\to 0}\limsup_{N\to\infty}\E|Err(\ep,N,\delta)|=0.
\]
Shifting the arguments of $\phi(t,\cdot), \psi(t,\cdot)$ in \eqref{artificial} by $z_1$ and $z_2$, respectively (with the latter two in the support of $\eta^\delta$), it can be shown that we cause an error of $O(\delta)$ in expectation, whereby we rewrite \eqref{artificial}
\begin{align*}
&\int_0^T\iint_{\R^{2d}} \frac{1}{N^2}\sum_{j,k}\theta^\ep\left(x_j(t)-x_k(t)-z_1+z_2\right)\eta^\delta(z_1)\eta^\delta(z_2)\phi(t,x_j(t)-z_1)\psi(t,x_k(t-z_2))dz_1dz_2dt+ Err_1(\ep,N, \delta)\\
&=\frac{1}{N^2}\int_0^T\iint_{\R^{2d}}\theta^\ep(w_1-w_2) \phi(t,w_1)\psi(t, w_2)\sum_{j\le N(t)}\eta^\delta\left(x_j(t)-w_1\right)\sum_{k\le N(t)}\eta^\delta\left(x_k(t)-w_2\right)dw_1dw_2dt + Err_1(\ep,N,\delta)\\
&= \int_0^T\int_{\R^{2d}} \theta^\ep(w_1-w_2)\phi(t,w_1)\psi(t, w_2)\left(\eta^\delta*_x\xi^N\right)(t, w_1) \left(\eta^\delta*_x\xi^N\right)(t, w_2)dw_1dw_2dt+Err_1(\ep,N, \delta),
\end{align*}
where the second line is a change of the order of integration, and
\[
\E|Err_1(\delta,N,\ep)|=\E|Err(\delta,N,\ep)|+O(\delta).
\] 
Since $\eta^\delta, \phi, \psi$ are all smooth, and $|w_1-w_2|<2C_0\ep$ within the support of $\theta^\ep$, changing the $w_2$ to $w_1$ in the argument of $\eta^\delta$ and $\psi(t,\cdot)$ can be shown to cause an error on the order $O(\ep\delta^{-2d-1})$ in expectation, and we can rewrite the above further
\begin{align*}
&\int_0^Tdt\int_{\R^d}dw_1\left(\int_{\R^d}dw_2\; \theta^\ep(w_1-w_2)\right)\phi(t,w_1)\psi(t, w_1)\left(\eta^\delta*_x\xi^N\right)(t,w_1)^2+Err_2(\ep,N, \delta)\\
&=\int_0^Tdt\int_{\R^d}dw_1\phi(t,w_1)\psi(t, w_1)\left(\eta^\delta*_x\xi^N\right)(t, w_1)^2+Err_2(\ep,N, \delta)\,,
\end{align*}
since $\int\theta^\ep=1$, where 
\[
\E|Err_2(\ep,N, \delta)|=\E|Err_1(\ep,N, \delta)|+O(\ep\delta^{-2d-1}). 
\]
Since $\ep(N)\to 0$ with $N$,  $\E|Err_2(\ep,N, \delta)|$ vanishes with $N$ and $\delta$ in the right order. This completes the proof of Proposition \ref{product}.   $\QED$
\\\\
To complete the proof of Theorem \ref{thm-main}, another ingredient is the tightness of the sequence of measures $\{\cP^N\}_N$ in $\cP(\cM)$, and the properties of its weak subsequential limits, as discussed next.

\begin{lemma}\label{tight}
Let any $d\ge 1$ and $T$ finite. The sequence $\{\cP^N\}_N$ induced by $\{\omega\mapsto \xi^N(dt,dx,\omega)\}_N$ is tight in $\cP(\cM)$, hence relatively compact in the weak topology.
\end{lemma}
\begin{proof}
Firstly, note that subsets of the form 
\begin{align*}
K_A:=\left\{\mu\in \cM:\; \int_{[0,T]\times\R^d}\left(1+|x|\right)\mu(dt,dx)\le A\right\}
\end{align*}
are relatively compact in $\cM$. Indeed, uniformly for $\mu\in K_A$, we have that 
\begin{align*}
\mu\left(\left([0,T]\times\ovl\B(0,L)\right)^c\right)\le L^{-1}\int_{\left([0,T]\times\ovl\B(0,L)\right)^c}|x|\mu(dt,dx)\le L^{-1}A
\end{align*}
for any $L>0$. 
Secondly, by coupling to a system of pure proliferation of unit rate, by a similar proof as Proposition \ref{progenitor}, we can show that 
\begin{align*}
\E\left[\int_0^T\sum_{j\le N(t)}\left(1+|x_j(t)|\right)dt\right]\le C_*N_0
\end{align*}
for some constant $C_*=C_*(d,T, R)$ finite. Thus, for any $\ep>0$, we can find $A=A(\ep)$ such that uniformly for all $N$, 
\begin{align*}
\cP^N\left({\ovl{K}_A}^c\right)&\le \P\left(\int_{[0,T]\times\R^d}\left(1+|x|\right)\xi^N(dt,dx)>A\right)\\
&\le A^{-1}\E\left[\int_0^T\frac{1}{N}\sum_{j\le N(t)}\left(1+|x_j(t)|\right)dt\right]\le A^{-1}C_*\|u_0\|_{L^1}<\ep
\end{align*}
by Markov's inequality. This implies that the sequence $\{\cP^N\}_N$ is tight.
\end{proof}

\begin{lemma}\label{density}
Let $d\ge 1$ and $T$ finite. Any weak subsequential limit $\cP^*$ of the sequence $\{\cP^N\}_N$ is supported on the subset of $\cM$ consiting of measures that are absolutely continuous with respect to the Lebesgue measure on $[0,T]\times\R^d$, with density bounded by a deterministic constant.
\end{lemma}

\begin{proof}
We adapt the proof strategy of \cite[Lemmas 4.1, 4.2]{HR} to our case. Taking a smooth approximation $\{\psi_n\}_n$ to the function $(x-k-2)_+$, where $k=k(d,T,u_0)$ is a constant to be determined, we fix a $C^2$ function $\psi: \R\to\R_+$ that is non-decreasing, convex, with $\psi(0)=0$, $\psi''$ bounded, and $\psi'=0$ for $x\le k$ and $\psi'\le 1$ for $x>k$. For simplicity, we denote 
\[
f^\delta(t,x)dt:=\left(\xi^N*_x \eta^\delta\right)(t,x)=\frac{1}{N}dt\sum_{j\le N(t)}\eta^\delta\left(t,x-x_j(t)\right),
\]
where $\eta$ is a smooth bump function as in \eqref{def:eta}.

By It\^o formula applied to the process $\int_{\R^d}\psi\left(f^\delta(t,x)\right)dx$, and then taking expectation, the martingale vanishes and we get that 
\begin{align}\label{sub-density}
&\E\int\psi\left(f^\delta(T, x)\right)dx\le \E\int\psi\left(f^\delta(0, x)\right)dx+\E\int_0^T\int\sum_{j\le N(t)}\frac{1}{2}\Delta_{x_j}\psi \left(f^\delta(t,x)\right)dxdt\nonumber\\
&\quad\quad\quad  +\E\int_0^T\int\sum_{j\le N(t)}\left[\psi\left(f^\delta(t,x)+\frac{1}{N}\eta^\delta(x-x_j(t)))\right)-\psi(f^\delta(t,x))\right]\, dxdt\nonumber\\
&\le \E\int\psi(f^\delta(0, x))dx+C\frac{N^{-1}}{\delta^{d+2}}\int|\nabla\eta|^2dx +\E\int_0^T\int\frac{1}{N}\sum_{j\le N(t)}\eta^\delta(x-x_j(t))1_{\{f^\delta(t,x)>k\}}dxdt\nonumber\\
&= \E\int\psi(f^\delta(0, x))dx+C\frac{N^{-1}}{\delta^{d+2}}\int|\nabla\eta|^2dx+\E\int_0^T\int f^\delta(t,x)1_{\{f^\delta(t,x)>k\}}dxdt,
\end{align}
where the analysis of the second term is the same as done in \cite[page 46]{HR}, in particular utilizing the identity \cite[(4.7)]{HR}. We argue that the \abbr{RHS} of \eqref{sub-density} converges to zero as $N\to\infty$ for some constant $k=k(T,d, u_0)$ and every fixed $\delta$. With that we can get the desired conclusion by repeating the argument of \cite[Lemma 4.2]{HR}. In particular, the constant $k+2$ is the upper bound on the density.

To this end, fixing $t$ and $x$ and we consider $\B(x, \delta)$, the open $\delta$-ball around $x$. We denote by $Z(t,x, \delta)$ the number of particles in an auxiliary binary Branching Brownian motion (\abbr{BBM}) of unit branching rate that fall into $\B(x, \delta)$ at time $t$, when starting with a single particle at $t=0$ distributed with density $(\int u_0)^{-1}u_0$. 

Recall that our particle system starts with $N_0=N\int u_0$ number of independent points with density $(\int u_0)^{-1}u_0$, and each particle generates its own lineage, stochastically dominated from above by a \abbr{BBM}. If we denote $Z^{(i)}(t,x, \delta)$ the number of particles in our proliferation system that fall into $\B(x,\delta)$ at time $t$ that come from the $i$-th lineage, for $i=1,2,...,N_0$, then $Z^{(i)}$ are dominated by i.i.d. copies $Z_i$ of $Z$.

Note that each $Z_i$ is a Poisson variable with constant mean $e^T$. By Chernoff's bound for sums of independent Poisson variables, for fixed $\delta, t, x$, some $C=C(d, T, u_0)$, any $N$ and $s>\E Z$ large enough, we have the tail estimates
\begin{align}
&\P\Big(f^\delta(t,x)\ge s\frac{\delta^{-d}||\eta||_\infty}{\int u_0}\Big)\le \P\Big(N_0^{-1}\sum_{i=1}^{N_0}Z^{(i)}\ge s\Big) \nonumber\\
&\le\P\Big(N_0^{-1}\sum_{i=1}^{N_0}Z_i\ge s\Big)\le C\exp\{-CN(s-\E Z)\log(s-\E Z)\}\label{concentration}
\end{align}
where the first inequality comes from unraveling the definition of $f^\delta$.
Further, due to the joint continuity of Brownian motion transition densities, $(t,x)\mapsto\E [Z(t,x, \delta)]$ is continuous. By the boundedness and compactly supportedness of $u_0$, this quantity decays to zero as $|x|\to\infty$, uniformly on $[0,T]$. Further, the limit as $\delta\to0$ of $\delta^{-d}\E [Z(t,x,\delta)]$ exists and is (up to a constant) the occupation density of a \abbr{BBM} at $(t,x)$. Thus, the following constant is universal:
\begin{align*}
k_1(d,T, u_0):=\sup_{t\in[0,T],\; x\in\R^d,\,\delta\in(0,1]}\left\{\delta^{-d}\E [Z(t,x, \delta)]\right\}<\infty.
\end{align*}
In particular, upon choosing 
\[
k=k(d,T,u_0):=2\gamma\vee \frac{2k_1||\eta||_\infty}{\int u_0}
\]
we have by \eqref{concentration} that for some $C=C(k)$ finite,
\[
\E \left[f^\delta(t,x)1_{\{f^\delta(t,x)>k\}}\right]=\int_k^\infty\P(f^\delta(t,x)>u)du
\le Ce^{-CN}.
\]
To conclude the proof, we note that the random variable $f^\delta(t,x)1_{\{f^\delta(t,x)>k\}}$ is dominated from above by $f^\delta(t,x)$ which is uniformly integrable. Indeed, it is simple to check that 
\[
\E \int_0^T\int f^\delta(t,x)dxdt\le e^T\int u_0, \quad \forall\delta, N>0.
\]
Thus, by the dominated convergence theorem, the third term of 
\eqref{sub-density} converges to zero as $N\to\infty$, for every fixed $\delta$. Since $k\ge 2\gamma$ where $\gamma=\|u_0\|_\infty$, we also have the first term of \eqref{sub-density} converging to zero. 
\end{proof}
\\\\
Now we complete the proof of Theorem \ref{thm-main}. Since the sequence $\{\cP^N\}_N$ is tight in $\cP(\cM)$, we can take a weakly converging subsequence $\cP^{N_k}\Rightarrow \cP^*$, and we denote by $\xi(dt,dx)$ a random variable taking values in $\cM$ that is distributed according to the measure $\cP^*$. For any $\iota>0$ we have that 
\begin{align*}
&\P\left(\left|\int u_0(x)\phi(0,x)dx +\left\langle \xi, (\partial_t+\frac{1}{2}\Delta+1)\phi\right\rangle-\left\langle \left(\xi*_x\eta^\delta\right)^2, \phi\right\rangle\right|>3\iota\right)\\
&\le \liminf_{N\to\infty}\P\left(\left|\int u_0(x)\phi(0,x)dx +\left\langle \xi^N, (\partial_t+\frac{1}{2}\Delta+1)\phi\right\rangle-\left\langle \left(\xi^N*_x\eta^\delta\right)^2, \phi\right\rangle\right|>3\iota\right)\\
&\le \liminf_{N\to\infty}\P\left(\left|Q^N(0,\eta(0))-\int \phi(0,x) u_0(x)dx\right|>\iota\right)+\liminf_{N\to\infty}\P\left(|\wt M_T|>\iota\right)\\
&\quad \quad \quad +\liminf_{N\to\infty}\P\left(\left|\left\langle (\xi^N*_x\eta^\delta)^2, \phi\right\rangle- \int_0^T \frac{1}{N^2}\sum_{j,k\le N(t)}\theta^\ep(x_j(t)-x_k(t))\phi(t, x_j(t))\, dt\right|>\iota\right)
\end{align*}
where we used the idenity \eqref{main-ito}.
We already have the first two limits equal to zero, and furthermore, by Proposition \ref{product}, we have that 
\begin{align*}
&\limsup_{\delta\to0}\limsup_{N\to\infty}\E\left|\int_0^T \frac{1}{N^2}\sum_{j,k\le N(t)}\theta^\ep(x_j(t)-x_k(t))\phi(t, x_j(t))\, dt-\left\langle(\xi^N*_x\eta^\delta)^2, \phi\right\rangle\right|=0
\end{align*}
By Lemma \ref{density}, we can write $\xi(dt,dx)=u(t,x)dtdx$ for some (possibly random) function $u(t,x)$ with a deterministic bound. With $u*_x\eta^\delta$ bounded uniformly in $\delta$, we have by the dominated convergence theorem that as $\delta\to 0$,
\begin{align*}
\E\left|\left\langle \phi(t,x), (u*_x\eta^\delta)(t,x)^2\right\rangle-\left\langle \phi(t,x), u(t,x)^2\right\rangle\right|\to 0.
\end{align*}
Taken together, we get that with probability one,
\[
\int u_0(x)\phi(0,x)dx +\left\langle u, (\partial_t+\frac{1}{2}\Delta+1)\phi\right\rangle-\left\langle u^2, \phi\right\rangle=0
\]
holds, whereby $u(t,x)$ is a weak solution of the \abbr{F-KPP} equation, in the sense of \eqref{weak-informal}. Since we will prove in Section \ref{uniqueness} that weak solutions to the \abbr{F-KPP} equation are unique, the limit $\cP^*$ must be unique and is a Dirac measure on $u(t,x)dtdx$. This completes the proof of our main result.

\section{Uniqueness of weak solutions of \abbr{F-KPP} equation}\label{uniqueness}
Denote by $L_{+}^{1}\left(  \mathbb{R}^{d}\right)  $ the set of nonnegative
integrable functions.

As a preliminary, let us recall that we have denoted by $\xi^N\left(
dt,dx\right)  $ the space-time empirical measure and we have remarked that it
has converging subsequences. Thus assume that $\xi^{N_k}\left(  dt,dx\right)
$ weakly converges to a space-time finite measure $\xi\left(  dt,dx\right)  $:%
\[
\lim_{k\rightarrow\infty}\int_{0}^{T}\int_{\mathbb{R}^{d}}K\left(  t,x\right)
\xi^{N_k}\left(  dt,dx\right)  =\int_{0}^{T}\int_{\mathbb{R}^{d}}K\left(
t,x\right)  \xi\left(  dt,dx\right)
\]
for every bounded continuous function $K$. Moreover, we know that there exists%
\[
u\in L^{\infty}\left(  \left[  0,T\right]  ;L^{\infty}\left(  \mathbb{R}%
^{d}\right)  \cap L_{+}^{1}\left(  \mathbb{R}^{d}\right)  \right)
\]
such that
\[
\int_{0}^{T}\int_{\mathbb{R}^{d}}K\left(  t,x\right)  \xi\left(  dt,dx\right)
=\int_{0}^{T}\int_{\mathbb{R}^{d}}K\left(  t,x\right)  u\left(  t,x\right)
dxdt.
\]
Moreover, we also know that the finite measure $\mu_{0}^{N}\left(  dx\right)
$ defined on $\mathbb{R}^{d}$ by  $\mu_{0}^{N}\left(  dx\right)  =\frac{1}%
{N}\sum_{j}\delta_{x_{j}\left(  0\right)  }$ converges weakly to a finite
measure $\mu_{0}\left(  dx\right)  $ with density $u_{0}\in L_{+}^{1}\left(
\mathbb{R}^{d}\right)  $:%
\[
\lim_{n\rightarrow\infty}\int_{\mathbb{R}^{d}}\phi\left(  x\right)  \mu
_{0}^{N}\left(  dx\right)  =\int_{\mathbb{R}^{d}}\phi\left(  x\right)
u_{0}\left(  x\right)  dx
\]
for every bounded continuous function $\phi$. Finally, let us also introduce
the notation $\mu_{t}^{N}\left(  dx\right)  $ for the family of finite
measures on $\mathbb{R}^{d}$, indexed by $t$, such that
\[
\xi^N\left(  dt,dx\right)  =\mu_{t}^{N}\left(  dx\right)  dt
\]
namely, more explicitly, $\mu_{t}^{N}\left(  dx\right)  =\frac{1}{N}\sum
_{j}\delta_{x_{j}\left(  t\right)  }\left(  dx\right)  $. We have, for $K\in
C_{c}^{1,2}\left(  \left[  0,T\right]  \times\mathbb{R}^{d}\right)  $ (compact
support is here a restriction only in space, since we take $\left[
0,T\right]  $ closed),%
\begin{align*}
\left\langle \mu_{T}^{N},K\left(  T\right)  \right\rangle  & =\left\langle
\mu_{0}^{N},K\left(  0\right)  \right\rangle +\int_{0}^{T}\left\langle
g_{t}^{N},\left(  \partial_{t}+\frac{1}{2}\Delta+1\right)  K\left(  t\right)
\right\rangle dt\\
& -N^{-2}\int_{0}^{T}\sum_{j,k}\theta^\epsilon\left(  x_{j}\left(
t\right)  -x_{k}\left(  t\right)  \right)    K\left(  t,x_{j}\left(
t\right)  \right)  dt+M_{T}.
\end{align*}
Notice that we may express the first time integral by means of $\xi^N\left(
dt,dx\right)  $, but we cannot do the same for the term $\left\langle \mu
_{T}^{N},K\left(  T\right)  \right\rangle $, since this is not a space-time
integral. Therefore we have%
\begin{align*}
\left\langle \mu_{T}^{N},K\left(  T\right)  \right\rangle  & =\left\langle
\mu_{0}^{N},K\left(  0\right)  \right\rangle +\int_{0}^{T}\int_{\mathbb{R}%
^{d}}\left(  \partial_{t}+\frac{1}{2}\Delta+1\right)  K\left(  t,x\right)
\xi^N\left(  dt,dx\right)  \\
& -N^{-2}\int_{0}^{T}\sum_{j,k}\theta^\epsilon\left(  x_{j}\left(
t\right)  -x_{k}\left(  t\right)  \right)   K\left(  t,x_{j}\left(
t\right)  \right)  dt+M_{T}.
\end{align*}
Now, from the convergence property of $\xi^{N_k}$ above, we have%
\begin{align*}
& \lim_{k\rightarrow\infty}\int_{0}^{T}\int_{\mathbb{R}^{d}}\left(
\partial_{t}+\frac{1}{2}\Delta+1\right)  K\left(  t,x\right)  \xi^{N_k}\left(
dt,dx\right)  \\
& =\int_{0}^{T}\int_{\mathbb{R}^{d}}\left(  \partial_{t}+\frac{1}{2}%
\Delta+1\right)  K\left(  t,x\right)  u\left(  t,x\right)  dxdt.
\end{align*}
Similarly, we have
\[
\lim_{k\rightarrow\infty}\left\langle \mu_{0}^{N_{k}},K\left(  0\right)
\right\rangle =\int_{\mathbb{R}^{d}}K\left(  0,x\right)  u_{0}\left(
x\right)  dx.
\]
The last term is the one carefully analyzed in the paper. But we cannot state
anything about the convergence of $\left\langle \mu_{T}^{N_{k}},K\left(
T\right)  \right\rangle $, unless we investigate the tightness of the emprical
measures $\mu_{\cdot}^{N}$ in $\cD\left(  \left[  0,T\right]  ;M_{+}\left(
\mathbb{R}^{d}\right)  \right)  $, a fact that we prefer to avoid. 

Therefore the only choice is to assume that
\[
K\left(  T\right)  =0.
\]
We deduce%
\begin{align*}
0  & =\int_{\mathbb{R}^{d}}K\left(  0,x\right)  u_{0}\left(  x\right)  dx\\
& +\int_{0}^{T}\int_{\mathbb{R}^{d}}\left(  \partial_{t}+\frac{1}{2}%
\Delta+1\right)  K\left(  t,x\right)  u\left(  t,x\right)  dxdt\\
& -\lim_{k\rightarrow\infty}N^{-2}\int_{0}^{T}\sum_{j,k}\theta^
\epsilon\left(  x_{j}\left(  t\right)  -x_{k}\left(  t\right)  
\right)  K\left(  t,x_{j}\left(  t\right)  \right)  dt
\end{align*}
(using also the fact that the martingale goes to zero). This motivates the
next definition. 

\begin{definition}\label{weak-sol}
Assume $u_{0}\in L^{\infty}\left(  \mathbb{R}^{d}\right)  \cap L_{+}%
^{1}\left(  \mathbb{R}^{d}\right)  $. We say that
\[
u\in L^{\infty}\left(  \left[  0,T\right]  ;L^{\infty}\left(  \mathbb{R}%
^{d}\right)  \cap L_{+}^{1}\left(  \mathbb{R}^{d}\right)  \right)
\]
is a weak solution of the Cauchy problem%
\begin{align}
\partial_{t}u &  =\frac{1}{2}\Delta u+u\left(  1-u\right)  \label{KPP}\\
u|_{t=0} &  =u_{0}\nonumber
\end{align}
if
\begin{align*}
0  & =\int_{\mathbb{R}^{d}}K\left(  0,x\right)  u_{0}\left(  x\right)
dx+\int_{0}^{T}\int_{\mathbb{R}^{d}}\left(  \partial_{t}+\frac{1}{2}%
\Delta+1\right)  K\left(  t,x\right)  u\left(  t,x\right)  dxdt\\
& -\int_{0}^{T}\int_{\mathbb{R}^{d}}u^{2}\left(  t,x\right)  K\left(
t,x\right)  dxdt
\end{align*}
for all test functions $K\in C_{c}^{1,2}\left(  \left[  0,T\right]
\times\mathbb{R}^{d}\right)  $ such that $K\left(  T,\cdot\right)  =0$.
\end{definition}

A priori, a weak solution does not have continuity properties in time and thus
the value at zero is not properly defined. Implicitly it is defined by the
previous identity, but we can do better.

\begin{lemma}
Given $\phi\in C_{c}^{2}\left(  \mathbb{R}^{d}\right)  $, the measurable
bounded function $t\mapsto\left\langle u\left(  t\right)  ,\phi\right\rangle $
has a continuous modification, with value equal to $\left\langle u_{0}%
,\phi\right\rangle $ at time zero. Moreover, denoting by $t\mapsto\left\langle
u\left(  t\right)  ,\phi\right\rangle $ the continuous modification, we have%
\begin{equation}
\left\langle u\left(  t\right)  ,\phi\right\rangle =\left\langle u_{0}%
,\phi\right\rangle +\frac{1}{2}\int_{0}^{t}\left\langle u\left(  s\right)
,\Delta\phi\right\rangle ds+\int_{0}^{t}\left\langle u\left(  s\right)
\left(  1-u\left(  s\right)  \right)  ,\phi\right\rangle ds\label{weak}%
\end{equation}
for all $t\in\left[  0,T\right]  $.
\end{lemma}

\begin{proof}
Given $t_{0}\in\lbrack0,T)$, $h>0$ such that $t_{0}+h\leq T$ and $\phi\in
C_{c}^{2}\left(  \mathbb{R}^{d}\right)  $, consider the function $K\left(
t,x\right)  =\phi\left(  x\right)  \chi\left(  t\right)  $, where $\chi\left(
t\right)  $ is equal to 1 for $t\in\left[  0,t_{0}\right]  $, $1-\frac{1}%
{h}\left(  t-t_{0}\right)  $ for $t\in\left[  t_{0},t_{0}+h\right]  $, zero
for $t\in\left[  t_{0}+h,T\right]  $. This function is only Lipschitz
continuous in time but it is not difficult to approximate it by a $C^{1}$
function of time (this is not really needed, since Lipschitz continuity in
time of the test functions would be sufficient in the definition above). We
have $\partial_{t}K\left(  t,x\right)  =\phi\left(  x\right)  \chi^{\prime
}\left(  t\right)  $ where $\chi^{\prime}\left(  t\right)  $ is equal to zero
outside $\left[  t_{0},t_{0}+h\right]  $ and to $-h^{-1}$ inside, with only
lateral derivatives at $t_{0}$ and $t_{0}+h$. Using this test function above
we get%
\begin{align*}
0  & =\int_{\mathbb{R}^{d}}\phi\left(  x\right)  u_{0}\left(  x\right)
dx+\int_{0}^{T}\chi^{\prime}\left(  t\right)  \int_{\mathbb{R}^{d}}\phi\left(
x\right)  u\left(  t,x\right)  dxdt\\
& +\int_{0}^{T}\chi\left(  t\right)  \left(  \left\langle u\left(  t\right)
,\frac{1}{2}\Delta\phi\right\rangle +\left\langle u\left(  t\right)  \left(
1-u\left(  t\right)  \right)  ,\phi\right\rangle \right)  dt
\end{align*}
namely%
\begin{align*}
0  & =\int_{\mathbb{R}^{d}}\phi\left(  x\right)  u_{0}\left(  x\right)
dx-\frac{1}{h}\int_{t_{0}}^{t_{0}+h}v\left(  t\right)  dt\\
& +\int_{0}^{t_{0}}\left(  \left\langle u\left(  t\right)  ,\frac{1}{2}%
\Delta\phi\right\rangle +\left\langle u\left(  t\right)  \left(  1-u\left(
t\right)  \right)  ,\phi\right\rangle \right)  dt\\
& -\frac{1}{h}\int_{t_{0}}^{t_{0}+h}\left(  t-t_{0}\right)  \left(
\left\langle u\left(  t\right)  ,\frac{1}{2}\Delta\phi\right\rangle
+\left\langle u\left(  t\right)  \left(  1-u\left(  t\right)  \right)
,\phi\right\rangle \right)  dt
\end{align*}
where we have denoted by $v\left(  t\right)  $ the bounded measurable function
$\int_{\mathbb{R}^{d}}\phi\left(  x\right)  u\left(  t,x\right)  dx$. Since
$u\left(  t\right)  $ is bounded, the function equal to $\left(
t-t_{0}\right)  \left(  \left\langle u\left(  t\right)  ,\frac{1}{2}\Delta
\phi\right\rangle +\left\langle u\left(  t\right)  \left(  1-u\left(
t\right)  \right)  ,\phi\right\rangle \right)  $ for $t\in\left[
t_{0},T\right]  $ and equal to zero at $t_{0}$ is continuous at $t=t_0$, hence%
\[
\lim_{h\rightarrow0}\frac{1}{h}\int_{t_{0}}^{t_{0}+h}\left(  t-t_{0}\right)
\left(  \left\langle u\left(  t\right)  ,\frac{1}{2}\Delta\phi\right\rangle
+\left\langle u\left(  t\right)  \left(  1-u\left(  t\right)  \right)
,\phi\right\rangle \right)  dt=0.
\]
By Lebesgue differentiability theorem, the following limit exists for a.e.
$t_{0}$:%
\[
\lim_{h\rightarrow0}\frac{1}{h}\int_{t_{0}}^{t_{0}+h}v\left(  t\right)
dt=v\left(  t_{0}\right)  .
\]
Therefore we get%
\[
v\left(  t_{0}\right)  =\int_{\mathbb{R}^{d}}\phi\left(  x\right)
u_{0}\left(  x\right)  dx+\int_{0}^{t_{0}}\left(  \left\langle u\left(
t\right)  ,\frac{1}{2}\Delta\phi\right\rangle +\left\langle u\left(  t\right)
\left(  1-u\left(  t\right)  \right)  ,\phi\right\rangle \right)  dt
\]
for a.e. $t_{0}$. The right-hand-side of this identity is a continuous
function of $t_{0}$, hence the function $v$ has a continuous modification. And
its value at $t_{0}=0$ is $\int_{\mathbb{R}^{d}}\phi\left(  x\right)
u_{0}\left(  x\right)  dx$.
\end{proof}

\medskip
We can now prove the main result of this section.

\begin{proposition}\label{unique}
Two weak solutions of the Cauchy problem (\ref{KPP}) coincide a.s.
\end{proposition}

\begin{proof}
\textbf{Step 1}. Let $u$ be a weak solution. Let $e^{t\frac{1}{2}\Delta}$ be the heat
semigroup, defined for instance on bounded measurable functions. In this step
we are going to prove that
\begin{equation}
u\left(  t\right)  =e^{t\frac{1}{2}\Delta}u_{0}+\int_{0}^{t}e^{\left(
t-s\right)  \frac{1}{2}\Delta}\left[  u\left(  s\right)  \left(  1-u\left(
s\right)  \right)  \right]  ds.\label{mild}%
\end{equation}
Let $\left(  \sigma_{\epsilon}\right)  _{\epsilon\in\left(  0,1\right)  }$ be
classical smooth compact support mollifiers; set%
\[
u_{\epsilon}\left(  t\right)  =\sigma_{\epsilon}\ast u\left(  t\right)  .
\]
Given $\psi\in C_{c}^{1,2}\left(  \mathbb{R}^{d}\right)  $, take $\phi
=\sigma_{\epsilon}^{-}\ast\psi$ in (\ref{weak}), where $\sigma_{\epsilon}%
^{-}\left(  x\right)  =\sigma_{\epsilon}\left(  -x\right)  $. Then, being
\begin{align*}
\left\langle u\left(  t\right)  ,\sigma_{\epsilon}^{-}\ast\psi\right\rangle
&  =\left\langle \sigma_{\epsilon}\ast u\left(  t\right)  ,\psi\right\rangle
=\left\langle u_{\epsilon}\left(  t\right)  ,\psi\right\rangle \\
\left\langle u_{0},\sigma_{\epsilon}^{-}\ast\psi\right\rangle  &
=\left\langle \sigma_{\epsilon}\ast u_{0},\psi\right\rangle \\
\left\langle u\left(  s\right)  ,\Delta\sigma_{\epsilon}^{-}\ast
\psi\right\rangle  &  =\left\langle u\left(  s\right)  ,\sigma_{\epsilon}%
^{-}\ast\Delta\psi\right\rangle =\left\langle u_{\epsilon}\left(  t\right)
,\Delta\psi\right\rangle =\left\langle \Delta u_{\epsilon}\left(  t\right)
,\psi\right\rangle
\end{align*}
we get%
\[
\left\langle u_{\epsilon}\left(  t\right)  ,\psi\right\rangle =\left\langle
\sigma_{\epsilon}\ast u_{0},\psi\right\rangle +\frac{1}{2}\int_{0}%
^{t}\left\langle \Delta u_{\epsilon}\left(  s\right)  ,\psi\right\rangle
ds+\int_{0}^{t}\left\langle \sigma_{\epsilon}\ast\left[  u\left(  s\right)
\left(  1-u\left(  s\right)  \right)  \right]  ,\psi\right\rangle ds
\]
and therefore%
\[
u_{\epsilon}\left(  t\right)  =\sigma_{\epsilon}\ast u_{0}+\frac{1}{2}\int%
_{0}^{t}\Delta u_{\epsilon}\left(  s\right)  ds+\int_{0}^{t}\sigma_{\epsilon
}\ast\left[  u\left(  s\right)  \left(  1-u\left(  s\right)  \right)  \right]
ds
\]
which implies that $t\mapsto u_{\epsilon}\left(  t,x\right)  $ is
differentiable, for every $x\in\mathbb{R}^{d}$. With classical arguments we
can rewrite the equation in the form%
\[
u_{\epsilon}\left(  t\right)  =e^{t\frac{1}{2}\Delta}\sigma_{\epsilon}\ast
u_{0}+\int_{0}^{t}e^{\left(  t-s\right)  \frac{1}{2}\Delta}\sigma_{\epsilon
}\ast\left[  u\left(  s\right)  \left(  1-u\left(  s\right)  \right)  \right]
ds.
\]
Notice that $e^{t\frac{1}{2}\Delta}$ is defined by a convolution with a smooth
kernel, for $t>0$, and thus by commutativity between convolutions we have
$e^{t\frac{1}{2}\Delta}\sigma_{\epsilon}\ast u_{0}=\sigma_{\epsilon}\ast
e^{t\Delta\frac{1}{2}}u_{0}$ and similarly under the integral sign. Hence we
can also write%
\[
u_{\epsilon}\left(  t\right)  =\sigma_{\epsilon}\ast e^{t\frac{1}{2}\Delta
}u_{0}+\int_{0}^{t}\sigma_{\epsilon}\ast e^{\left(  t-s\right)  \frac{1}%
{2}\Delta}\left[  u\left(  s\right)  \left(  1-u\left(  s\right)  \right)
\right]  ds.
\]
Given $\phi\in C_{c}\left(  \mathbb{R}^{d}\right)  $, we deduce%
\[
\left\langle u\left(  t\right)  ,\sigma_{\epsilon}^{-}\ast\phi\right\rangle
=\left\langle e^{t\frac{1}{2}\Delta}u_{0},\sigma_{\epsilon}^{-}\ast
\phi\right\rangle +\int_{0}^{t}\left\langle e^{\left(  t-s\right)  \frac{1}%
{2}\Delta}\left[  u\left(  s\right)  \left(  1-u\left(  s\right)  \right)
\right]  ,\sigma_{\epsilon}^{-}\ast\phi\right\rangle ds.
\]
Since $\sigma_{\epsilon}^{-}\ast\phi$ converges uniformly to $\phi$, from
dominated convergence theorem we deduce%
\[
\left\langle u\left(  t\right)  ,\phi\right\rangle =\left\langle e^{t\frac
{1}{2}\Delta}u_{0},\phi\right\rangle +\int_{0}^{t}\left\langle e^{\left(
t-s\right)  \frac{1}{2}\Delta}\left[  u\left(  s\right)  \left(  1-u\left(
s\right)  \right)  \right]  ,\phi\right\rangle ds
\]
and thus we get (\ref{mild}).

\textbf{Step 2}. Assume that $u^{\left(  i\right)  }$ are two weak solutions.
Then, from (\ref{mild}),
\[
u^{\left(  1\right)  }\left(  t\right)  -u^{\left(  2\right)  }\left(
t\right)  =\int_{0}^{t}e^{\left(  t-s\right)  k\Delta}\left(  u^{\left(
1\right)  }\left(  s\right)  -u^{\left(  2\right)  }\left(  s\right)  \right)
\left(  1-u^{\left(  1\right)  }\left(  s\right)  -u^{\left(  2\right)
}\left(  s\right)  \right)  ds
\]
hence%
\[
\left\Vert u^{\left(  1\right)  }\left(  t\right)  -u^{\left(  2\right)
}\left(  t\right)  \right\Vert _{\infty}\leq\int_{0}^{t}\left\Vert u^{\left(
1\right)  }\left(  s\right)  -u^{\left(  2\right)  }\left(  s\right)
\right\Vert _{\infty}\left(  1+\left\Vert u^{\left(  1\right)  }\left(
s\right)  \right\Vert _{\infty}+\left\Vert u^{\left(  1\right)  }\left(
s\right)  \right\Vert _{\infty}\right)  ds.
\]
Since, by assumption, $u^{\left(  i\right)  }$ are bounded, we deduce
$\left\Vert u^{\left(  1\right)  }\left(  t\right)  -u^{\left(  2\right)
}\left(  t\right)  \right\Vert _{\infty}=0$ by Gronwall lemma.
\end{proof}

\bibliographystyle{plain}
\bibliography{generator_v12_arXiv}

\begin{thebibliography}{10}

\bibitem{BM}
V.~Bansaye and S.~M\'{e}l\'{e}ard.
\newblock {\em Stochastic models for structured populations}, volume~1 of {\em
  Mathematical Biosciences Institute Lecture Series. Stochastics in Biological
  Systems}.
\newblock Springer, Cham; MBI Mathematical Biosciences Institute, Ohio State
  University, Columbus, OH, 2015.
\newblock Scaling limits and long time behavior.

\bibitem{CM}
N.~Champagnat and S.~M\'{e}l\'{e}ard.
\newblock Invasion and adaptive evolution for individual-based spatially
  structured populations.
\newblock {\em J. Math. Biol.}, 55(2):147--188, 2007.

\bibitem{DFL}
A.~De~Masi, P.~A. Ferrari, and J.~L. Lebowitz.
\newblock Reaction-diffusion equations for interacting particle systems.
\newblock {\em J. Statist. Phys.}, 44(3-4):589--644, 1986.

\bibitem{DFPV}
A.~De~Masi, T.~Funaki, E.~Presutti, and M.~E. Vares.
\newblock Fast-reaction limit for {G}lauber-{K}awasaki dynamics with two
  components.
\newblock {\em ALEA Lat. Am. J. Probab. Math. Stat.}, 16(2):957--976, 2019.

\bibitem{DP}
A.~De~Masi and E.~Presutti.
\newblock {\em Mathematical methods for hydrodynamic limits}, volume 1501 of
  {\em Lecture Notes in Mathematics}.
\newblock Springer-Verlag, Berlin, 1991.

\bibitem{FLT}
J.~Farf\'{a}n, C.~Landim, and K.~Tsunoda.
\newblock Static large deviations for a reaction-diffusion model.
\newblock {\em Probab. Theory Related Fields}, 174(1-2):49--101, 2019.

\bibitem{FL}
F.~Flandoli and M.~Leimbach.
\newblock Mean field limit with proliferation.
\newblock {\em Discrete Contin. Dyn. Syst. Ser. B}, 21(9):3029--3052, 2016.

\bibitem{FLO}
F.~Flandoli, M.~Leimbach, and C.~Olivera.
\newblock Uniform convergence of proliferating particles to the {FKPP}
  equation.
\newblock {\em J. Math. Anal. Appl.}, 473(1):27--52, 2019.

\bibitem{FT}
T.~Funaki and K.~Tsunoda.
\newblock Motion by mean curvature from {G}lauber-{K}awasaki dynamics.
\newblock {\em J. Stat. Phys.}, 177(2):183--208, 2019.

\bibitem{HR}
A.~Hammond and F.~Rezakhanlou.
\newblock The kinetic limit of a system of coagulating {B}rownian particles.
\newblock {\em Arch. Ration. Mech. Anal.}, 185(1):1--67, 2007.

\bibitem{KL}
C.~Kipnis and C.~Landim.
\newblock {\em Scaling limits of interacting particle systems}, volume 320 of
  {\em Grundlehren der Mathematischen Wissenschaften [Fundamental Principles of
  Mathematical Sciences]}.
\newblock Springer-Verlag, Berlin, 1999.

\bibitem{McK}
H.~P. McKean.
\newblock Application of {B}rownian motion to the equation of
  {K}olmogorov-{P}etrovskii-{P}iskunov.
\newblock {\em Comm. Pure Appl. Math.}, 28(3):323--331, 1975.

\bibitem{MR}
S.~M\'{e}l\'{e}ard and S.~Roelly-Coppoletta.
\newblock A propagation of chaos result for a system of particles with moderate
  interaction.
\newblock {\em Stochastic Process. Appl.}, 26(2):317--332, 1987.

\bibitem{Met}
M.~M\'{e}tivier.
\newblock Quelques probl\`emes li\'{e}s aux syst\`emes infinis de particules et
  leurs limites.
\newblock In {\em S\'{e}minaire de {P}robabilit\'{e}s, {XX}, 1984/85}, volume
  1204 of {\em Lecture Notes in Math.}, pages 426--446. Springer, Berlin, 1986.

\bibitem{NO}
G.~Nappo and E.~Orlandi.
\newblock Limit laws for a coagulation model of interacting random particles.
\newblock {\em Ann. Inst. H. Poincar\'{e} Probab. Statist.}, 24(3):319--344,
  1988.

\bibitem{Oel2}
K.~Oelschl\"{a}ger.
\newblock A law of large numbers for moderately interacting diffusion
  processes.
\newblock {\em Z. Wahrsch. Verw. Gebiete}, 69(2):279--322, 1985.

\bibitem{Oel}
K.~Oelschl\"{a}ger.
\newblock On the derivation of reaction-diffusion equations as limit dynamics
  of systems of moderately interacting stochastic processes.
\newblock {\em Probab. Theory Related Fields}, 82(4):565--586, 1989.

\bibitem{Ste}
A.~Stevens.
\newblock The derivation of chemotaxis equations as limit dynamics of
  moderately interacting stochastic many-particle systems.
\newblock {\em SIAM J. Appl. Math.}, 61(1):183--212, 2000.

\bibitem{Szn}
A.-S. Sznitman.
\newblock Topics in propagation of chaos.
\newblock In {\em \'{E}cole d'\'{E}t\'{e} de {P}robabilit\'{e}s de
  {S}aint-{F}lour {XIX}---1989}, volume 1464 of {\em Lecture Notes in Math.},
  pages 165--251. Springer, Berlin, 1991.

\bibitem{Uch}
K.~Uchiyama.
\newblock Pressure in classical statistical mechanics and interacting
  {B}rownian particles in multi-dimensions.
\newblock {\em Ann. Henri Poincar\'{e}}, 1(6):1159--1202, 2000.

\bibitem{Va}
S.~R.~S. Varadhan.
\newblock Scaling limits for interacting diffusions.
\newblock {\em Comm. Math. Phys.}, 135(2):313--353, 1991.

\end{thebibliography}

Authors address: Scuola Normale Superiore di Pisa. Piazza Dei Cavalieri 7. Pisa PI 56126. Italia. 

E-mails: \{franco.flandoli, ruojun.huang\}@ sns.it.

\end{document}